\documentclass[a4paper,11pt, leqno]{article}

\usepackage{latexsym,enumerate,graphicx}
\usepackage{amsmath,amsthm,amsopn,amstext,amscd,amsfonts,amssymb}
\usepackage[T1]{fontenc}
\usepackage{fullpage,yhmath}
\usepackage{hyperref}
\usepackage{eucal}
\usepackage{color}
\usepackage{verbatim}
\usepackage[normalem]{ulem}
\usepackage{frcursive}

\newenvironment{frcseries}{\fontfamily{frc}\selectfont}{}

\numberwithin{equation}{section}

\newtheorem{theorem}{Theorem}

\newtheorem{proposition}{Proposition}
\newtheorem{remark}{Remark}

\newtheorem{definition}{Definition}

\numberwithin{theorem}{section}
\numberwithin{corollary}{section}
\numberwithin{lemma}{section}
\numberwithin{definition}{section}
\numberwithin{proposition}{section}
\numberwithin{remark}{section}

\newcommand{\RR}{\mathbb R^N}
\newcommand{\C}{\mathbb C}
\newcommand{\Ze}{\mathbb Z}
\newcommand{\R}{\mathbb R}
\newcommand{\rn}{\mathbb R^N}

\newcommand{\medint}{-\kern  -,375cm\int}
\newcommand{\dint}{\displaystyle\int}

\newcommand{\di}{\mathop{\mathrm{d}\!}}

\newcommand{\dive}{\mathop{\mathrm{div}}}

\makeatletter
\@namedef{subjclassname@2020}{  \textup{2020} Mathematics Subject Classification}
\makeatother
\date{}

\begin{document}
\title{Kohler-Jobin inequality for $p$-Laplace operator\\ in the Gauss space}

\author{Francesco Chiacchio$^1$, Vincenzo Ferone$^{1,3}$, Anna Mercaldo$^1$, Jing Wang$^2$}
\footnotetext[1]
{{Dipartimento di Matematica e Applicazioni ``Renato Caccioppoli'', Universit\`a degli Studi di Napoli Federico II, Via Cintia, Complesso Universitario Monte S. Angelo, 80143 Napoli, Italy.}}

%
\footnotetext[2]
{{Department of Mathematics, Purdue University,
West Lafayette, Indiana, 47907, USA.}}
\footnotetext[3]
{{Corresponding author: ferone@unina.it}}
\maketitle

\setcounter{tocdepth}{1}

\begin{abstract}
A sharp lower bound for the first Dirichlet eigenvalue of the $p$-laplacian in Gaussian space is derived for sets with prescribed generalized  torsional rigidity. The result provides an extension of the classical spectral inequality due to Kohler-Jobin. The proof is based on a careful analysis of the generalized  torsional rigidity and on a sharp mass comparison result. Furthermore, a Payne-Rayner type inequality is established.


\end{abstract}


\section{Introduction}\label{intr}

Let us consider the first eigenvalue of the $p$-Laplacian in the Gauss space, also known as the Gaussian $p$-Laplacian. 
More precisely, for $p>1$, we study the following eigenvalue problem
\begin{equation} \label{eq.0}
\left\{
\begin{array}
[c]{lll}%
 -\dive(\phi_N (x)|Du|^{p-2}Du)=\lambda \phi_N (x)|u|^{p-2}u & & \text{in }%
\Omega,\\
\\
u=0 & & \text{on }\partial\Omega,
\end{array}
\right. %
\end{equation}
where
\begin{equation}\label{gauss}
\phi_N (x) =\frac 1{(2\pi)^{\frac N2}}\exp
\left(-\frac{|x|^2}2\right)
\end{equation}
and, here and throughout the paper,
 $\Omega $ is a possibly unbounded open set in 
 $ \mathbb{R}^N$
with $\gamma_N (\Omega)<1$, where 
$\gamma_N (A)$  
stands for the Gaussian measure of $A\subset\R^N$:
\[
\gamma_N (A) =\int_A\phi_N (x)\,\di x\in[0,1].
\]

%

Denoting by $L^p(\Omega,\phi_N)$ the weighted Lebesgue space and by $W_{0}^{1,p}(\Omega,\phi_N)$ the weighted Sobolev space, it is well known that the first eigenvalue $\lambda_1(\Omega)$ has the variational characterization
\begin{equation}\label{eigen}
\lambda_1(\Omega)=\min_{w\in W_{0}^{1,p} \left( \Omega,\phi _{N} \right)\backslash\{0\}}
\dfrac{\displaystyle\int_\Omega|Dw|^p\, \di \gamma_N}{\displaystyle\int_\Omega |w|^p\, \di \gamma_N}.
\end{equation}
Moreover, there exists a positive eigenfunction $u_1$ associated with $\lambda_1(\Omega)$ that attains the minimum in \eqref{eigen}.

Then we define the $p$-torsional rigidity by
\begin{equation}\label{TT}
T(\Omega)=\max_{w\in W_{0}^{1,p} \left( \Omega ,\phi_{N} \right)\backslash\{0\}}\dfrac{\displaystyle\left(\int_\Omega w\, \di \gamma_N\right)^p}{\displaystyle\int_\Omega|Dw|^p\, \di \gamma_N}.
\end{equation}
It is known that its maximum is attained at $w=v$, where $v$ is the torsion function that solves the following boundary value problem
\begin{equation} \label{ve}
\left\{
\begin{array}
[c]{lll}%
-\dive(\phi_N (x)|Dv|^{p-2}Dv)=\phi_N (x)& & \text{in }%
\Omega,\\
\\
v=0 & & \text{on }\partial\Omega.
\end{array}
\right.
\end{equation}
and it follows that 
\begin{equation}\label{TQ}
T(\Omega)=\left(\int_\Omega v(x)\, \di \gamma_N\right)^{p-1}.
\end{equation}

The extremal sets for the principal frequency and the torsional rigidity have been extensively studied in the literature. Particularly when $\Omega$ is bounded and $\phi_N (x)\equiv 1$, the differential operator in \eqref{eq.0}, \eqref{ve} reduces to the $p$-Laplacian. When $p=2$, these quantities are closely related to classical isoperimetric-type inequalities.
Among sets of given measure, the ball minimizes $\lambda_1(\Omega)$, as stated in the Lord Rayleigh conjecture and proven by Faber and Krahn (\cite{Fa,Kr}), and maximizes $T(\Omega)$, as stated in the Saint-Venant conjecture and proven by P\'olya (\cite{Po}). 
Moreover, in \cite{PS} P\'olya and Szeg\H o stated the stronger conjecture that among sets with fixed torsional rigidity, the ball minimizes the principal frequency. This conjecture was firstly proved by Kohler-Jobin in \cite{KMCb,KMZa} using a new rearrangement technique known as \emph{transplantation \`a integrales de Dirichlet \'egales}. 
For $p>1$, a nonlinear version of Kohler-Jobin inequality was established in \cite{Br} using a similar approach.

In this paper, we consider these isoperimetric problems for the Gaussian $p$-Laplacian. The motivation is not merely a passage to a weighted setting, but rather reflects a shift toward an infinite-dimensional geometry. It is well knwon that the Gaussian measure plays the role of a canonical reference measure in infinite-dimensional settings, and can be viewed as the limit of normalized Lebesgue measure on high-dimensional spheres under finite-dimensional projections. Half-spaces then appear as the limiting counterparts of the Euclidean extremal sets, namely the Euclidean balls.

As regards the Gaussian $p$-Laplacian, the case $p=2$ was studied in \cite{HL} (see also
\cite{CGNT}).
In \cite{HL}, it was shown that for any domain $\Omega\subset\R^N$, if $H$ is a half-space such that $T(H)=T(\Omega)$, then
$$
\lambda_1(\Omega)\ge\lambda_1(H).
$$
However, the proof in \cite{HL} relies heavily on specific properties of certain special functions rather than general rearrangement arguments. In particular, for general $p>1$ where explicit solutions are no longer available, this method does not seem to extend directly.
We therefore adopt a different approach that is based on a variational formulation of a \emph{generalized torsional rigidity}, defined as follows. For $\alpha \in \R$, let
\begin{equation}
    \label{torgen_max_intro}
Q_p(\alpha, \Omega)=
\sup_{w\in 
W_{0}^{1,p} \left( \Omega ,\phi_{N} \right)}
\left\{-\int_\Omega|Dw(x)|^p\, \di \gamma_N+\alpha\int_\Omega |w(x)|^p\, \di \gamma_N+p\int_\Omega w(x)\, \di \gamma_N\right\}.
\end{equation}
This notion was first introduced in \cite{Ba} in the case  $p=2$ and $\phi_N (x)\equiv 1$.
For any \(\alpha \in (-\infty, \lambda_1(\Omega))\),
the maximum in \eqref{torgen_max_intro} is attained at the generalized torsion function \({v}\), which solves the problem
\begin{equation}
    \label{ve_intro}
\left\{
\begin{array}
[c]{lll}%
-\dive(\phi_N (x)|Dv|^{p-2}Dv)=\alpha v^{p-1}\phi_N (x)+\phi_N (x)& & \text{in }%
\Omega,\\
\\
v=0 & & \text{on }\partial\Omega.
\end{array}
\right. 
\end{equation}
In particular, when $\alpha=0$, the generalized torsional rigidity reduces to the $p$-torsional rigidity, that is,
\begin{equation}\label{TQa}
T(\Omega)=\left(\dfrac{Q_p(0,\Omega)}{p-1}\right)^{p-1}.
\end{equation}

In this paper we will prove that, for any \(\alpha \in (-\infty, \lambda_1(\Omega))\) and for any  set $\Omega\subset\RR$ with $\gamma_N (\Omega)<1$, it holds
\begin{equation}\label{tor-alpha}
\lambda_1(\Omega) \geq \lambda_1(H_{\alpha})\quad\text{where}\ \ H_{\alpha}\ \ \text{is a half-space s.t.} \ Q_p(\alpha, H_{\alpha})=Q_p(\alpha, \Omega).
\end{equation}
Clearly, when \(\alpha = 0\), the above statement implies the inequality proven in \cite{HL} in the case $p=2$, so it states a more general result also in the linear case. Among other things, our proof is based on the fact that the mapping $\alpha \mapsto H_{\alpha}$ in \eqref{tor-alpha} is monotone with respect to the inclusion and the full statement \eqref{tor-alpha} follows taking the limit as $\alpha\rightarrow\lambda_1(\Omega)$. A key step in proving the quoted monotonicity is a comparison result in terms of a mass comparison inequality between the solution $v$ to problem \eqref{ve_intro} and the function which solves the same problem on the half-space $H_{\alpha}$. Let us emphasize that the properties of the mapping $\alpha \mapsto H_{\alpha}$ are obtained after a careful analysis of the dependence of $Q_p(\alpha, \Omega)$ on $\alpha$ and on $\Omega$. In particular, it allows us to prove a derivation formula with respect to $\alpha$ which turns out to be very useful in the proof of the main result. Let us finally mention that the needed properties of $Q_p(\alpha, \Omega)$ in the special case where $\Omega$ is a half-space are proven without using the explicit form of the generalized torsion
function.

Let us observe that, when the Gaussian weight is absent, a similar approach has been adopted in \cite{KMCa,KMZa,KMZb}, where, using different techniques, the counterpart of \eqref{tor-alpha} is proven. Unfortunately, since \eqref{tor-alpha} is obtained via a limit procedure, it seems that the method does not give a characterization of the equality case.

Let us finally remark that the method used to prove the comparison result for the function $v$ can be easily adapted to prove a comparison result for the first eigenfuntion for problem \eqref{eq.0}. This observation allows us to prove a Payne-Rayner type inequality for the Gaussian
$p$-laplacian extending results obtained, for example, in \cite{PR1}, \cite{PR2}, \cite{K}, \cite{C2}, \cite{AFT}, \cite{BCF}.

The paper is structured as follows. 
In Section \ref{s1}, we introduce some notation and we collect some preliminary results about the first eigenvalue and the torsional rigidity defined in \eqref{eigen} and \eqref{TT}, respectively. In Section \ref{s2} we prove some properties of the generalized torsional rigidity defined in \eqref{torgen_max_intro}, both in a general domain and in a 
half-space.
In Section \ref{s3} we prove the comparison result and the Payne-Rayner type inequality mentioned above, while in Section \ref{s4} we prove the main result. Finally, Section \ref{s5} is dedicated to show that, when $p=2$, it is possible to solve explicitly problem \eqref{ve_intro} on a half-space, and Section \ref{s6} contains some proofs of auxiliary results stated in the previous sections.

\section{Preliminary results}\label{s1}
This section is devoted to recalling basic properties of weighted rearrangements with respect to the Gaussian measure, as well as classical results on the first eigenvalue and the 
$p$-torsional rigidity in the Gauss space.

\subsection{Symmetrization with respect to Gauss measure}

Let $u$ be a measurable function defined on a subset $\Omega \subseteq \mathbb{R}^n$.
The distribution function of $u$, denoted by $\mu$, is defined as the Gaussian 
measure of the superlevel sets of $|u|$, i.e., the map from $[0, +\infty)$ into 
$[0, \gamma_N(\Omega)]$ given by
$$
\mu
    (t)=\gamma_N \left (\left \{x\in \Omega :~ ~|u(x)|>t\right \}\right ).
$$
The function $\mu$ is non-increasing and right-continuous. 
The decreasing
rearrangement of $u$ is defined as
\begin{equation}\label{riord} 
u^*(s)=\inf \{ t \ge 0 : \mu(t) \le s\}, \qquad 0 < s \le \gamma_N (\Omega). 
\end{equation}
It follows from the definition that $u^*(s)$ is also non-increasing and right-continuous.

The rearrangement of $u$ with respect to the Gaussian measure is the function
$u^\sharp$ whose superlevel sets $\{u^\sharp > t\}$ are half-spaces having the
same Gaussian measure as the corresponding superlevel sets of $|u|$. It is clear that one can choose the half-spaces in various ways. In order to fix the notation, here and in the following, we consider half-spaces such that the scalar product of the unit vector $(1,0,\dots,0)$ with the external normal to their boundary is positive. More precisely, for any set $E\subset\R^N$ we denote by $E^\sharp$ the
half-space in the form
$$
E^\sharp=\{x=(x_1,x_2,\dots,x_N)\in\R^N\,:\,\,x_1<\tau\},
$$
having the same Gauss measure of $E$, that is, $\gamma_N(E^\sharp)=\gamma_N(E)$.
A straightforward calculation gives
\begin{equation*}
\tau=k^{-1}(\gamma_N(\Omega)),
\end{equation*}
where $k(\sigma)$ is the function 
\begin{equation}
k(\sigma)=\gamma_N\left (\{x\in \rn:~~x_1<\sigma\}\right )=\frac {1}{\sqrt
{2\pi}}
\int_{-\infty}^{\sigma}
\exp\left({-\frac{t^2}{2}}\right) \di t \label{phi}
\end{equation}
for all $\sigma \in \R$.
Then $u^\sharp $  is a map from
$\Omega^\sharp$ in
$[0,+\infty)$ defined by
$$
u^\sharp (x)=u^\sharp (x_1) = u^*\left (k (x_1)\right ).
$$
Weighted $L^p$-norm  is invariant with respect to such a rearrangement, that is:
\begin{equation}\label{inv}
 ||u||_{L^{p}(\Omega,\phi_N)} =
||u^\sharp||_{L^p(\Omega^{\sharp},\phi_N)}.
\end{equation}
for any $1\le p\le \infty$.

\noindent Moreover, a P\'olya-Sz\'eg\H{o} principle also holds true with respect to the Gaussian measure. 
 \cite{E,T, T2}.

\begin{theorem} \label{polya_szego}
 Let  $p>1$ and 
 $u\in W^{1,p}_0(\Omega,\phi_N )$. Then $u^\sharp \in W^{1,p}_0(\Omega^\sharp,\phi_N)$ and 
\begin{equation*}
\int_{\Omega}|D u|^{p}\di\gamma_N \geq \int_{\Omega^\sharp}|D u^{\sharp}|^{p}\di\gamma_N.  \label{Polya_Szego_Eh2}
\end{equation*}
\end{theorem}

\noindent We recall that a Hardy-Littlewood type inequality holds true 
\begin{equation}
\int_\Omega |f(x)g(x)| \di\gamma_N\le
\int_{\Omega^\sharp} f^\sharp (x) g^\sharp (x)\di\gamma_N=\int_0^{\gamma_N(\Omega ) }
f^*(s)g^*(s)\, \di s\label{Har}
\end{equation}
Finally we recall the following result (cf. \cite{ChongRice}).
\begin{proposition} \label{relazione}
Let $f,g\in L_+^{1}(\Omega,\phi_N )$. Then the following statements are equivalent:
\begin{equation}
 \int_0^t f^*(s)\,\di s\le
\int_0^t g^*(s)\,\di s\,, \quad \forall t\in [0, \gamma_N(\Omega )]
\end{equation}
\begin{equation}
 \int_\Omega F(f)\,\di\gamma_N\le
 \int_\Omega F(g)\,\di\gamma_N\,, 
\end{equation}
for all convex, nonnegative  Lipschitz continuo functions $F$ such that $F(0)=0$.
\end{proposition}


\subsection{First eigenvalue and torsional rigidity}

We recall the definition of weighted Sobolev space.
\begin{definition}
 The weighted Sobolev space $W^{1,p}(\Omega, \phi_N)$ is the set of all
functions
 $\psi \in W^{1,1}_{loc}(\Omega)$ such that $(\psi,|D\psi|) \in L^p(\Omega, \phi_N)\times
L^p(\Omega,\phi_N)$,
 endowed with the norm
 $$
 ||\psi||_{W^{1,p}(\Omega,\phi_N)} =
||\psi||_{L^p(\Omega,\phi_N)}+||D\psi||_{L^p(\Omega,\phi_N)}.
 $$
  The weighted Sobolev space $W_0^{1,p}(\Omega,\phi_N)$
  is the closure of $C_0^\infty(\Omega)$ in $W^{1,p}(\Omega, \phi_N)$.  
  \end{definition}

\noindent For $p>1$ and $\Omega\subset\R^N$, with $\gamma_N (\Omega)<1$, we consider the following eigenvalue problem
\begin{equation*}
\left\{
\begin{array}
[c]{lll}%
 -\dive(\phi_N (x)|Du|^{p-2}Du)=\lambda \phi_N (x)|u|^{p-2}u & & \text{in }%
\Omega,\\
\\
u=0 & & \text{on }\partial\Omega,
\end{array}
\right. %
\end{equation*}
It is well known that the first eigenvalue $\lambda_1(\Omega)$ has the variational characterization
\begin{equation*}
\lambda_1(\Omega)=\min_{w\in W_{0}^{1,p} \left( \Omega,\phi _{N} \right)\backslash\{0\}}
\frac{\|Dw\|^p_{L^p(\Omega,\phi_N)}}{\displaystyle\|w\|^p_{L^p(\Omega,\phi_N)}},
\end{equation*}
and that there exists a positive function $u_1$, which is an eigenfunction corresponding to $\lambda_1(\Omega)$, attaining the minimum in \eqref{eigen}
Note that Theorem \ref{polya_szego}, together with the fact that symmetrization preserves $L^p$ norms (see, e.g. \cite{BCF})  ensures the validity of the following Faber-Krahn inequality.

\begin{theorem}\label{KR} Let
$\Omega $ be an open set in 
 $ \mathbb{R}^N$,
with $\gamma_N (\Omega)<1$. Then
\begin{equation}\label{FK}
\lambda_1(\Omega) \ge \lambda_1(\Omega^\sharp).
\end{equation}    
\end{theorem}
Arguing as in the case of the first eigenvalue, Theorem \ref{polya_szego} implies that the following isoperimetric inequality holds for the $p$-torsional rigidity defined in \eqref{TT}.

\begin{theorem}\label{SV}
Let
$\Omega $ be an open set in 
 $ \mathbb{R}^N$,
with $\gamma_N (\Omega)<1$. Then
\begin{equation}\label{SVe}
T(\Omega) \le T(\Omega^\sharp).
\end{equation}    
\end{theorem}

We add some results concerning the properties of the first eigenvalue on a half-space $H_t$, $t\in\R$,
$$
H_t=\{x=(x_1,x_2,\dots,x_N)\in\R^N\,:\,\,x_1<t\}.
$$
We firstly observe that, when one considers the first eigenvalue on $H_t$, the use of Theorem \ref{KR} implies that the first eigenfunction depends only on the variable $x_1$, so we have
\begin{equation}
\label{1dim}
\lambda_1(H_t)
=\min_{w\in W^{1,p}_{0}
\left(
(-\infty,t),
 \phi_{1} 
\right)\setminus\{0\}}
\frac{\dint_{-\infty}^{t} |w'(\sigma)|^{p}\,\phi_{1}(\sigma)\,\di\sigma}
{\dint_{-\infty}^{t} |w(\sigma)|^{p}\,\phi_{1}(\sigma)\,\di\sigma},
\end{equation}
and the first eigenfunction $u>0$, which achieves the minimum above, 
solves the problem
\begin{equation} \label{2eq}
\left\{
\begin{array}
[c]{lll}
-\bigl(|u'|^{p-2}u'\phi_{1}(\sigma)\bigr)'=\lambda_1(H_t)u^{p-1}\phi_{1}(\sigma)& & \sigma\in
(-\infty,t),\\
\\
u\in W^{1,p}_{0}\!
\left(
(-\infty,t),\phi_{1}\right). & & 
\end{array}
\right.
\end{equation}
It is clear that $\lambda_1(H_t)$ is a decreasing function with respect to $t$. Moreover, the
following result holds true.
\begin{proposition}\label{half} Let $H_t$, $t\in\R$, be the half-space defined above.
There exists a positive function $u_1=u_1(x_1)$, depending only on $x_1$, which is an eigenfunction corresponding to $\lambda_1(H_t)$ and
 $\lambda_1(H_t)$ is simple, that is, if $u$ is a solution to problem \eqref{2eq}, 
then $u=\beta u_1$, with $\beta \in \R$. Furthermore the mapping $t\mapsto \lambda_1(H_t)$ is decreasing and, denoting by $u$ the positive eigenfunction such that $||u||_{L^p(H_t,\phi_N)}=1$, the following differentiation formula holds true
\begin{equation}\label{shape}
\frac\di{\di t}\lambda_1(H_t)=-(p-1)|u'(t)|^p\phi_1(t).
\end{equation}    
\end{proposition}

Let us observe that formula \eqref{shape} can be seen as a shape derivative of the first eigenvalue in a case which, apparently, has been not treated in the literature (see \cite{BHMP} for the case $p=2$), so, in Section \ref{s6} we give a direct proof in the case of half-spaces.

\section{A generalized torsional rigidity}\label{s2}

Let $\Omega $ be an open  set of $ \R^N$. For $\alpha\in(-\infty,\lambda_1(\Omega))$ we consider the following generalization of torsional rigidity
\begin{equation}\label{QQ}
Q_p(\alpha, \Omega)=\sup_{w\in W_{0}^{1,p} \left( \Omega ,\phi_{N} \right)}\left\{-\int_\Omega|Dw(x)|^p\, \di \gamma_N+\alpha\int_\Omega |w(x)|^p\, \di \gamma_N+p\int_\Omega w(x)\, \di \gamma_N\right\}.
\end{equation}
As we will see,
the maximum of the functional
\begin{equation}
\label{F}
F_\alpha(w):=-\int_{\Omega }|Dw(x)|^{p}\,\di\gamma _{N}+\alpha \int_{\Omega
}|w(x)|^{p}\,\di\gamma _{N}+p\int_{\Omega }w(x)\di\gamma _{N}\,,
\end{equation}
is achieved just
for $w=v$, where $v$ is the 
unique weak solution (see Proposition \ref{prop} to problem
\begin{equation} \label{veq.0}
\left\{
\begin{array}
[c]{lll}%
-\dive(\phi_N (x)|Dv|^{p-2}Dv)=\alpha \phi_N (x)|v|^{p-2}v+\phi_N (x)& & \text{in }%
\Omega,\\
\\
v=0 & & \text{on }\partial\Omega\,.
\end{array}
\right. 
\end{equation}
This means that
\[
 v\in W_{0}^{1,p} \left( \Omega ,\phi_{N} \right),
 \]
 and
 \begin{equation}\label{vweak}
  \int_{\Omega} |D v|^{p-2}D vD \varphi \,\di\gamma_N= \alpha \int_{\Omega}|  v|^{p-2}  v \varphi \,\di\gamma_N +\int_{\Omega}  \varphi\, \di\gamma_N\,, 
  \notag
 \end{equation}
 for every $\varphi \in W_0^{1,p}(\Omega, \phi_N)$.

\noindent An easy calculation proves that the following equality holds true:
\begin{equation}\label{Q}
Q_p(\alpha, \Omega)=(p-1)\int_\Omega v(x)\, \di \gamma_N.
\end{equation}
As observed in the introduction, when $\alpha=0$, the generalized torsional rigidity reduces to the $p$-torsional rigidity in the sense that \eqref{TQa}
holds true.

We now derive some useful properties of $Q_p(\alpha, \Omega)$ which are consequences of its  definition.
\begin{proposition}\label{prop}
Let $\Omega\subset \R^N$, with $\gamma_N (\Omega)<1$. Then
\begin{enumerate}[\rm(a)]
\item\label{(a)} $\displaystyle Q_p(\alpha, \Omega)$ is finite$\quad\iff\quad 
-\infty<\alpha<\lambda_1(\Omega);$

\item\label{(b)}
for any  $\alpha < \lambda _{1} \left( \Omega  \right)$
the functional $F_\alpha(w)$, defined in \eqref{F},
 has a unique maximizer 
 $v\ge0$, which  is the unique weak solution to problem \eqref{veq.0},
and therefore \eqref{Q} holds true; 

\item\label{(c)} $\displaystyle Q_p(\alpha, \Omega)$ is increasing with respect to $\Omega$ (in the sense of inclusion), that is,
\[\Omega_1\subset\Omega_2\quad\Longrightarrow \quad 
Q_p(\alpha,\Omega_1)
\le
Q_p(\alpha,\Omega_2);
\]
\item\label{(d)}  if $\alpha<\lambda_1(\Omega^\sharp)$\,, then 
\[Q_p(\alpha,\Omega)\le Q_p(\alpha,\Omega^\sharp).
\]
\end{enumerate}
\end{proposition}
\begin{proof} We prove in sequence the various items.
\vskip.2cm
\noindent \underline{Item (\ref{(a)})}
\vskip.3cm

{  

Let us suppose that $-\infty<\alpha<\lambda_1(\Omega)$.
 For any  $ w \in W_{0}^{1,p} \left( \Omega ,\phi_{N} \right)$ we have
\begin{equation*}
F_{\alpha}(w)\leq \left( \alpha -\lambda _{1}\left( \Omega\right) \right)\int_{\Omega }|w(x)|^{p}\,\di\gamma _{N}+p\int_{\Omega
}\left\vert w(x)\right\vert \di\gamma _{N}.
\end{equation*}
For $\varepsilon>0$, Young inequality yields
\begin{equation*}
F_{\alpha}(w)\leq \left( \alpha+\varepsilon -
\lambda _{1} \left( \Omega  \right) \right)
\int_{\Omega }|w(x)|^{p}\,\di\gamma _{N}
+C(\varepsilon)\gamma_N (\Omega),\qquad\forall w \in
W_{0}^{1,p} \left( \Omega ,\phi_{N} \right)
\end{equation*}
and, if $\varepsilon$ is sufficiently small, we have
\begin{equation*}
F_{\alpha}(w)\leq C,\qquad
\forall w \in 
W_{0}^{1,p} \left( \Omega ,\phi_{N} \right),
\end{equation*}
which proves that $Q_p(\alpha,\Omega)$ is finite.

%
We now suppose that $Q_p(\alpha,\Omega)$ is finite. Let us assume, by contradiction, that
\begin{equation}
\text{$\alpha \geq \lambda _{1}$}\left( \text{$\Omega $}\right) .  \label{a>}
\end{equation}
Consider the family of functions $\left\{ tw_{1}\right\} $
where $t>0$ and $w_{1}$ is the positive eigenfunction corresponding to $\lambda _{1}$$\left( \text{$\Omega $}
\right) $, such that
\begin{equation*}
\left\Vert w_{1}\right\Vert _{L^{p}(\Omega,\phi_{N})}=1.
\end{equation*}
Using (\ref{a>}), we have
\begin{equation*}
F_{\alpha}\left( tw_{1}\right) =-t^{p}\int_{\Omega }|Dw_{1}(x)|^{p}\,\di\gamma
_{N}+\alpha t^{p}\int_{\Omega }w_{1}(x)^{p}\,\di\gamma _{N}+pt\int_{\Omega
}w_{1}(x)\di\gamma _{N}
\end{equation*}

\begin{equation*}
=pt\int_{\Omega }w_{1}(x)\di\gamma _{N}+t^{p}\left( \alpha -\text{$\lambda _{1}
$}\left( \text{$\Omega $}\right) \right)\int_{\Omega }w_{1}(x)^{p}\,\di\gamma _{N} \geq pt\int_{\Omega
}w_{1}(x)\di\gamma _{N}.
\end{equation*}
It follows that
\begin{equation*}
\lim_{t\rightarrow +\infty }F_{\alpha}\left( tw_{1}\right) =+\infty 
\end{equation*}
and therefore
\begin{equation*}
Q_{p}(\alpha ,\Omega ) =+\infty .
\end{equation*}
This is a contradiction, as we are under the assumption that 
$Q_{p}(\alpha ,\Omega )$ is finite.

\vskip.2cm
\noindent \underline{Item (\ref{(b)})}
\vskip.3cm

We begin by proving that for all $ \alpha < \lambda_{1}\left( \Omega \right)$
the functional $F$ attains a maximum.  We  assume that $\alpha >0$, since, otherwise, 
the proof becomes even simpler.

From the previous considerations, we know that
\begin{align}
\sup_{w\in W_{0}^{1,p} (\Omega ,\phi_{N})} F_{\alpha}(w) &= \sup_{w\in W_{0}^{1,p} (\Omega ,\phi_{N})} \left\{ -\int_{\Omega} |Dw(x)|^{p}\,\di\gamma_{N} \right. \label{l>} \\
&\quad \left. +\alpha \int_{\Omega }|w(x)|^{p}\,\di\gamma_{N} +p\int_{\Omega }w(x)\di\gamma_{N}\right\} < + \infty. \nonumber
\end{align}
Let $\{w_{k}\}_{k\in \mathbb{N}}\subset W_{0}^{1,p} \left( \Omega ,\phi_{N} \right)$ be a maximizing sequence. 
By (\ref{l>}), there exists a constant $C$ such that
\begin{equation*}
\int_{\Omega }|Dw_{k}(x)|^{p}\,\di\gamma _{N}-\alpha \int_{\Omega}
|w_{k}(x)|^{p}\,\di\gamma _{N}-p\int_{\Omega }w_{k}(x)\di\gamma _{N}\leq C
\qquad \forall k\in \mathbb{N}.
\end{equation*}
Hence, for every $k\in\mathbb{N}$ we obtain
$$
\int_{\Omega }|Dw_{k}(x)|^{p}\,\di\gamma _{N} 
\leq \frac{\alpha}{\lambda _{1}\left( \Omega \right)}
       \int_{\Omega }|Dw_{k}(x)|^{p}\,\di\gamma _{N}
      +p\int_{\Omega }|w_{k}(x)|\,\di\gamma _{N}+C.
$$
 Young inequality ensures that  $\forall \varepsilon >0$ there exists 
$C_{\varepsilon }>0$ such that
\begin{equation*}
\left(1-\frac{\alpha}{\lambda _{1}\left( \Omega \right)}\right)
\int_{\Omega }|Dw_{k}(x)|^{p}\,\di\gamma _{N}
\leq \varepsilon \int_{\Omega }|w_{k}(x)|^{p}\,\di\gamma _{N}
   +C_{\varepsilon}\gamma _{N}(\Omega )+C.
\end{equation*}
Finally, using the continuous embedding of 
$W_{0}^{1,p}\left( \Omega ,\di\gamma _{N}\right)$
into 
$L^{p}\left( \Omega ,\di\gamma _{N}\right)$, the arbitrariness of 
$\varepsilon$ implies that
\begin{equation}
\int_{\Omega }|Dw_{k}(x)|^{p}\,\di\gamma _{N}\leq C
\qquad \forall k\in\mathbb{N},  \label{Equi}
\end{equation}
where, here and in the sequel, $C$ denotes a constant whose value may vary 
from line to line, but which does not depend on the significant parameters 
of the problem.

Finally, from the compact embedding of 
$W_{0}^{1,p}\left( \Omega ,\phi_{N}\right)$
into
$L^{p}\left( \Omega ,\gamma _{N}\right)$
one deduces that, up to a not relabelled subsequence, 
there exists a function $w\in W_{0}^{1,p}\left( \Omega ,\phi_{N}\right)$ 
such that
\begin{equation*}
\left\{ 
\begin{array}{cc}
w_{k}\rightarrow w & \text{weakly in }W_{0}^{1,p}\left( \Omega ,\phi_{N}\right), \\
\\
w_{k}\rightarrow w & \text{strongly in }L^{p}\left( \Omega ,\phi_{N}\right).
\end{array}
\right.
\end{equation*}
From this, in a standard way, the claim follows immediately. Hence, since $F_{\alpha}$ attains its 
maximum at $w$, its Euler equation in \eqref{veq.0} admits at least one solution.

We now address the issues related to uniqueness. First of all we observe that any solution to problem \eqref{veq.0} is nonnegative. Indeed, using $v_-=\max\{-v,0\}$ as test function in \eqref{veq.0}, we have
\begin{equation*}
\int_{\{v<0\}}|Dv|^p\, \di \gamma_N-\alpha\int_{\{v<0\}} |v|^p\, \di \gamma_N=\int_{\{v<0\}} v\, \di \gamma_N\,.
\end{equation*}
Since the right-hand side is negative and $\alpha<\lambda_1(\Omega)$, this   gives a contradiction if $v_-\not \equiv0$.

The uniqueness of the solution to problem \eqref{veq.0} can be proven using an argument which goes back to \cite{BO} (for the case $p=2$) and which has been used in \cite{DS} (and subsequently refined in \cite{Lin}).

Let $u\not\equiv v$ be two solutions to problem \eqref{veq.0}. For $0<\sigma<k<+\infty$ we consider the test functions
\begin{equation}\label{test}
\varphi_1=u+\sigma-\frac{\bigl(T_k(v+\sigma)\bigr)^p}{(u+\sigma)^{p-1}},\qquad
\varphi_2=v+\sigma-\frac{\bigl(T_k(u+\sigma)\bigr)^p}{(v+\sigma)^{p-1}},
\end{equation}
where $T_k(s)=\min\{s,k\}$ denotes the usual truncation function. It is immediate to observe that $\varphi_1, \varphi_2\in
W_{0}^{1,p} \left( \Omega ,\phi_{N} \right)$, so we can use $\varphi_1$ in the equation satisfied by $u$ and $\varphi_2$ in the equation satisfied by $v$. Using the notation $u_\sigma=u+\sigma$, $v_\sigma=v+\sigma$, we obtain
\begin{eqnarray*}
\int_\Omega|Du|^p\, \di \gamma_N-p\int_\Omega|Du|^{p-2}DuD\bigl(T_k(v_\sigma)\bigr)\frac{\bigl(T_k(v_\sigma)\bigr)^{p-1}}{(u_\sigma)^{p-1}}\, \di \gamma_N
+(p-1)\int_\Omega|Du|^{p}\frac{\bigl(T_k(v_\sigma)\bigr)^{p}}{(u_\sigma)^{p}}\, \di \gamma_N=\\
=\alpha\int_\Omega \left(\frac u{u_\sigma}\right)^{p-1}\left((u_\sigma)^p-\bigl(T_k(v_\sigma)\bigr)^p\right)\, \di \gamma_N+\int_\Omega \left(u_\sigma-\frac{\bigl(T_k(v_\sigma)\bigr)^p}{(u_\sigma)^{p-1}}\right)\, \di \gamma_N,
\\
\int_\Omega|Dv|^p\, \di \gamma_N-p\int_\Omega|Dv|^{p-2}DvD\bigl(T_k(u_\sigma)\bigr)\frac{\bigl(T_k(u_\sigma)\bigr)^{p-1}}{(v_\sigma)^{p-1}}\, \di \gamma_N
+(p-1)\int_\Omega|Dv|^{p}\frac{\bigl(T_k(u_\sigma)\bigr)^{p}}{(v_\sigma)^{p}}\, \di \gamma_N=\\
=\alpha\int_\Omega \left(\frac v{v_\sigma}\right)^{p-1}\left((v_\sigma)^p-\bigl(T_k(u_\sigma)\bigr)^p\right)\, \di \gamma_N+\int_\Omega \left(v_\sigma-\frac{\bigl(T_k(u_\sigma)\bigr)^p}{(v_\sigma)^{p-1}}\right)\, \di \gamma_N.
\end{eqnarray*}
Summing the above equalities and using Young inequalities
\begin{equation*}
\left|p|Du|^{p-2}DuD\bigl(T_k(v_\sigma)\bigr)\frac{\bigl(T_k(v_\sigma)\bigr)^{p-1}}{(u_\sigma)^{p-1}}\right|
\le |D\bigl(T_k(v_\sigma)\bigr)|^p+(p-1)|Du|^{p}\frac{\bigl(T_k(v_\sigma)\bigr)^{p}}{(u_\sigma)^{p}}
\end{equation*}
\begin{equation*}
\left|p|Dv|^{p-2}DvD\bigl(T_k(u_\sigma)\bigr)\frac{\bigl(T_k(u_\sigma)\bigr)^{p-1}}{(v_\sigma)^{p-1}}\right|
\le |D\bigl(T_k(u_\sigma)\bigr)|^p+(p-1)|Dv|^{p}\frac{\bigl(T_k(u_\sigma)\bigr)^{p}}{(v_\sigma)^{p}}
\end{equation*}
we obtain
\begin{eqnarray*}
\alpha\left(\int_\Omega \left(\frac u{u_\sigma}\right)^{p-1}\left((u_\sigma)^p-\bigl(T_k(v_\sigma)\bigr)^p\right)\, \di \gamma_N
+\int_\Omega \left(\frac v{v_\sigma}\right)^{p-1}\left((v_\sigma)^p-\bigl(T_k(u_\sigma)\bigr)^p\right)\, \di \gamma_N\right)+\\
+\int_\Omega \left(u_\sigma-\frac{\bigl(T_k(v_\sigma)\bigr)^p}{(u_\sigma)^{p-1}}\right)\, \di \gamma_N
+\int_\Omega \left(v_\sigma-\frac{\bigl(T_k(u_\sigma)\bigr)^p}{(v_\sigma)^{p-1}}\right)\, \di \gamma_N\\
\ge\int_\Omega|Du|^p-\int_\Omega|D\bigl(T_k(u_\sigma)\bigr)|^p\, \di \gamma_N+\int_\Omega|Dv|^p-\int_\Omega|D\bigl(T_k(v_\sigma)\bigr)|^p\, \di \gamma_N\ge0.
\end{eqnarray*}
Therefore we have:
\begin{eqnarray*}
\alpha\left(\int_\Omega \left(\frac u{u_\sigma}\right)^{p-1}\left((u_\sigma)^p-\bigl(T_k(v_\sigma)\bigr)^p\right)\, \di \gamma_N
+\int_\Omega \left(\frac v{v_\sigma}\right)^{p-1}\left((v_\sigma)^p-\bigl(T_k(u_\sigma)\bigr)^p\right)\, \di \gamma_N\right)+\\
+\int_\Omega \left(u_\sigma-\frac{\bigl(T_k(v_\sigma)\bigr)^p}{(u_\sigma)^{p-1}}\right)\, \di \gamma_N
+\int_\Omega \left(v_\sigma-\frac{\bigl(T_k(u_\sigma)\bigr)^p}{(v_\sigma)^{p-1}}\right)\, \di \gamma_N\ge0.
\end{eqnarray*}
We now pass to the limit as $k$ goes to $+\infty$ by monotone convergence  and we obtain:
\begin{eqnarray}\label{passig}
\alpha\int_\Omega \left(\left(\frac u{u_\sigma}\right)^{p-1}-\left(\frac v{v_\sigma}\right)^{p-1}\right)\left((u_\sigma)^p-(v_\sigma)^p\right)\, \di \gamma_N\ge\qquad\qquad\\
\notag\ge\int_\Omega \left(\left(\frac 1{v_\sigma}\right)^{p-1}-\left(\frac 1{u_\sigma}\right)^{p-1}\right)\left((u_\sigma)^p-(v_\sigma)^p\right)\, \di \gamma_N.
\end{eqnarray}
Now we pass to the limit as $\sigma  $ goes to zero in the last inequality. 
Firstly, since, 
\[
\lim_{\sigma\to 0}\left(\left(\frac u{u_\sigma}\right)^{p-1}-\left(\frac v{v_\sigma}\right)^{p-1}\right)=0\,, \qquad \hbox{a. e. in }\Omega\,,
\]
and, for $\sigma<k$, 
\[
\left |\left(\left(\frac u{u_\sigma}\right)^{p-1}-\left(\frac v{v_\sigma}\right)^{p-1}\right)\left((u_\sigma)^p-(v_\sigma)^p\right)\right |\le 2(u^p+v^p+2^pk^p)
\,,
\]
Lebesgue dominated convergence theorem implies
\begin{equation}\label{leb}
\lim_{\sigma\rightarrow0^+}\int_\Omega \left(\left(\frac u{u_\sigma}\right)^{p-1}-\left(\frac v{v_\sigma}\right)^{p-1}\right)\left((u_\sigma)^p-(v_\sigma)^p\right)\, \di \gamma_N=0,
\end{equation}
Therefore, for $\sigma \to 0$, inequality \eqref{passig} gives
\begin{equation}\label{abs}
\liminf_{\sigma\rightarrow0^+}\int_\Omega \left(\left(\frac 1{v_\sigma}\right)^{p-1}-\left(\frac 1{u_\sigma}\right)^{p-1}\right)\left((u_\sigma)^p-(v_\sigma)^p\right)\, \di \gamma_N
\le0,
\end{equation}
Moreover, since 
\begin{eqnarray*}
\left(\left(\frac 1{v_\sigma}\right)^{p-1}-\left(\frac 1{u_\sigma}\right)^{p-1}\right)\left((u_\sigma)^p-(v_\sigma)^p\right)\ge0,
\end{eqnarray*}
if $u\not\equiv v$, by Fatou lemma we have
\begin{eqnarray*}
\liminf_{\sigma\rightarrow0^+}\int_\Omega \left(\left(\frac 1{v_\sigma}\right)^{p-1}-\left(\frac 1{u_\sigma}\right)^{p-1}\right)\left((u_\sigma)^p-(v_\sigma)^p\right)\, \di \gamma_N\ge\\
\ge\int_\Omega \left(\left(\frac 1{v}\right)^{p-1}-\left(\frac 1{u}\right)^{p-1}\right)\left(u^p-v^p\right)\, \di \gamma_N>0\,.
\end{eqnarray*}
This inequality contradicts \eqref{abs}, and the claim is proved.
\vskip.2cm
\noindent \underline{Item (\ref{(c)})}
\vskip.3cm
The claim follows immediately from the definition of $Q_p(\alpha, \Omega)$.

\vskip.2cm
\noindent \underline{Item (\ref{(d)})}
\vskip.3cm
The claim follows immediately from P\'olya-Szeg\H o principle given by Theorem \ref{polya_szego} and property \eqref{inv}. 
}
\end{proof}
Now we prove some properties of $Q_p(\alpha, \Omega)$ when $\Omega$ is fixed and the parameter $\alpha$ varies.
\begin{proposition}\label{propa}
Let $\Omega\subset \R^N$, with $\gamma_N (\Omega)<1$. Then:
\begin{enumerate}[\rm(a)]
\item\label{(aa)} if $v_{(\alpha)}$ denotes the solution to problem \eqref{veq.0} for a given value of the parameter $\alpha\in (-\infty,\lambda_1(\Omega))$, we have:
\begin{equation*}
v_{(\alpha)}(x)\le v_{(\beta)}(x), \qquad x\in \Omega,\>-\infty<\alpha<\beta<\lambda_1(\Omega);
\end{equation*}
\item\label{(ab)} $\quad\displaystyle Q_p(\alpha, \Omega)$ is increasing with respect to $\alpha$ and, if $v$ solves \eqref{veq.0}, it holds
\[\dfrac\di{\di\alpha}Q_p(\alpha,\Omega)=\int_\Omega |v(x)|^p\, \di\gamma_N;
\]
\item\label{(ac)} $\quad\displaystyle\lim_{\alpha\rightarrow-\infty} Q_p(\alpha,\Omega)=0;$
\item\label{(ad)} $\quad\displaystyle\lim_{\alpha\rightarrow \lambda_1(\Omega)^-} Q_p(\alpha,\Omega)=+\infty.$
\end{enumerate}
\end{proposition}
\begin{proof} We prove in sequence the various items.
\vskip.2cm
\noindent \underline{Item (\ref{(aa)})}
\vskip.3cm
The claim can be proved proceeding as in the proof of item (\ref{(b)}) of Proposition \ref{prop}.

Let us fix $-\infty<\alpha<\beta<\lambda_1(\Omega)$ and let us put
\begin{equation*}
u=v_{(\alpha)},\qquad w=v_{(\beta)}.
\end{equation*}
For $0<\sigma<k<+\infty$ we consider the test functions
\begin{equation*}
\varphi_1=\frac{\left(\bigl(T_k(u+\sigma)\bigr)^p-{\bigl(T_k(w+\sigma)\bigr)^p}\right)_+}{(u+\sigma)^{p-1}},\qquad
\varphi_2=\frac{\left(\bigl(T_k(u+\sigma)\bigr)^p-{\bigl(T_k(w+\sigma)\bigr)^p}\right)_+}{(w+\sigma)^{p-1}},
\end{equation*}
where, for $s\in\R$, we use the notation $T_k(s)=\min\{s,k\}$, $s_+=\max\{s,0\}$. It is immediate to observe that $\varphi_1, \varphi_2\in
W_{0}^{1,p} \left( \Omega ,\phi_{N} \right)$, so we can use $\varphi_1$ in the equation satisfied by $u$ and $\varphi_2$ in the equation satisfied by $w$. Using the notation $u_\sigma=u+\sigma$, $w_\sigma=w+\sigma$ and $E_k=\{x:T_k(u_\sigma(x))> T_k(w_\sigma(x))\}$, we obtain
\begin{eqnarray*}
\int_{E_k}|Du|^{p-2}DuD\left(\frac{\bigl(T_k(u_\sigma)\bigr)^p}{u_\sigma^{p-1}}\right)\, \di \gamma_N-p\int_{E_k}|Du|^{p-2}DuD\bigl(T_k(w_\sigma)\bigr)\frac{\bigl(T_k(w_\sigma)\bigr)^{p-1}}{(u_\sigma)^{p-1}}\, \di \gamma_N+
\\
+(p-1)\int_{E_k}|Du|^{p}\frac{\bigl(T_k(w_\sigma)\bigr)^{p}}{(u_\sigma)^{p}}\, \di \gamma_N=\\
=\alpha\int_{E_k} \left(\frac u{u_\sigma}\right)^{p-1}\left(\bigl(T_k(u_\sigma)\bigr)^p-\bigl(T_k(w_\sigma)\bigr)^p\right)\, \di \gamma_N+\int_{E_k} \left(\frac{(T_k(u_\sigma))^p-\bigl(T_k(w_\sigma)\bigr)^p}{(u_\sigma)^{p-1}}\right)\, \di \gamma_N,
\end{eqnarray*}
\begin{eqnarray*}
\int_{E_k}|Dw|^{p-2}DwD\left(\frac{\bigl(T_k(w_\sigma)\bigr)^p}{w_\sigma^{p-1}}\right)\, \di \gamma_N-p\int_{E_k}|Dw|^{p-2}DwD\bigl(T_k(u_\sigma)\bigr)\frac{\bigl(T_k(u_\sigma)\bigr)^{p-1}}{(w_\sigma)^{p-1}}\, \di \gamma_N+\\
+(p-1)\int_{E_k}|Dw|^{p}\frac{\bigl(T_k(u_\sigma)\bigr)^{p}}{(w_\sigma)^{p}}\, \di \gamma_N=\\
=\beta\int_{E_k} \left(\frac w{w_\sigma}\right)^{p-1}\left(\bigl(T_k(w_\sigma)\bigr)^p-\bigl(T_k(u_\sigma)\bigr)^p\right)\, \di \gamma_N+\int_{E_k} \left(\frac{(T_k(w_\sigma))^p-\bigl(T_k(u_\sigma)\bigr)^p}{(w_\sigma)^{p-1}}\right)\, \di \gamma_N.
\end{eqnarray*}

Summing the above equalities and using Young inequalities
\begin{equation*}
\left|p|Du|^{p-2}DuD\bigl(T_k(w_\sigma)\bigr)\frac{\bigl(T_k(w_\sigma)\bigr)^{p-1}}{(u_\sigma)^{p-1}}\right|
\le |D\bigl(T_k(w_\sigma)\bigr)|^p+(p-1)|Du|^{p}\frac{\bigl(T_k(w_\sigma)\bigr)^{p}}{(u_\sigma)^{p}}
\end{equation*}
\begin{equation*}
\left|p|Dw|^{p-2}DwD\bigl(T_k(u_\sigma)\bigr)\frac{\bigl(T_k(u_\sigma)\bigr)^{p-1}}{(w_\sigma)^{p-1}}\right|
\le |D\bigl(T_k(u_\sigma)\bigr)|^p+(p-1)|Dw|^{p}\frac{\bigl(T_k(u_\sigma)\bigr)^{p}}{(w_\sigma)^{p}}
\end{equation*}
we obtain
\begin{align}\label{Tk}
\alpha\int_{E_k} \left(\left(\frac u{u_\sigma}\right)^{p-1}-\left(\frac w{w_\sigma}\right)^{p-1}\right)\left(\bigl(T_k(u_\sigma)\bigr)^p-\bigl(T_k(w_\sigma)\bigr)^p\right)\, \di \gamma_N&\\
+(\beta-\alpha)\int_{E_k} \left(\frac w{w_\sigma}\right)^{p-1}\left(\bigl(T_k(w_\sigma)\bigr)^p-\bigl(T_k(u_\sigma)\bigr)^p\right)\, \di \gamma_N&\notag\\
+\int_{E_k} \left(\frac{(T_k(u_\sigma))^p-\bigl(T_k(w_\sigma)\bigr)^p}{(u_\sigma)^{p-1}}\right)\, \di \gamma_N
+\int_{E_k} \left(\frac{(T_k(w_\sigma))^p-\bigl(T_k(u_\sigma)\bigr)^p}{(w_\sigma)^{p-1}}\right)\, \di \gamma_N&\ge\notag\\
\ge\int_{E_k}|Du|^{p-2}DuD\left(\frac{\bigl(T_k(u_\sigma)\bigr)^p}{u_\sigma^{p-1}}\right)\, \di \gamma_N-\int_{E_k}|D\bigl(T_k(u_\sigma)\bigr)|^p\, \di \gamma_N&+\notag\\
+\int_{E_k}|Dw|^{p-2}DwD\left(\frac{\bigl(T_k(w_\sigma)\bigr)^p}{w_\sigma^{p-1}}\right)\, \di \gamma_N-\int_{E_k}|D\bigl(T_k(w_\sigma)\bigr)|^p\, \di \gamma_N.&\notag
\end{align}
We can pass to the limit as $k$ goes to $+\infty$ on the left-hand side by dominated convergence, 
in view of the fact that $\chi_{E_k}$ converges pointwise to $\chi_{u>w}$. As regards the integrals on the right-hand side of \eqref{Tk}, we have
\begin{align*}
&\lim_{k\rightarrow+\infty}\left(\int_{E_k}|Du|^{p-2}DuD\left(\frac{\bigl(T_k(u_\sigma)\bigr)^p}{u_\sigma^{p-1}}\right)\, \di \gamma_N-\int_{E_k}|D\bigl(T_k(u_\sigma)\bigr)|^p\, \di \gamma_N\right)=\\
&\quad=-\lim_{k\rightarrow+\infty}(p-1)\int_{E_k\cap\{u_\sigma>k\}}|Du|^{p}\frac{k^p}{u_\sigma^{p}}\, \di \gamma_N=0,
\end{align*}
with an analogous result for the terms on the right-hand side of \eqref{Tk} which contain $w$.

Then we get:

\begin{eqnarray*}
\alpha\int_{u>w} \left(\left(\frac u{u_\sigma}\right)^{p-1}-\left(\frac w{w_\sigma}\right)^{p-1}\right)\left((u_\sigma)^p-(w_\sigma)^p\right)\, \di \gamma_N+\\
+(\beta-\alpha)\int_{u>w} \left(\frac w{w_\sigma}\right)^{p-1}\left((w_\sigma)^p-(u_\sigma)^p\right)\, \di \gamma_N\ge\\
\ge\int_{u>w} \left(\left(\frac 1{w_\sigma}\right)^{p-1}-\left(\frac 1{u_\sigma}\right)^{p-1}\right)\left((u_\sigma)^p-(w_\sigma)^p\right)\, \di \gamma_N.
\end{eqnarray*}
It is clear that 
\begin{equation*}
\lim_{\sigma\rightarrow0^+}\int_{u>w} \left(\left(\frac u{u_\sigma}\right)^{p-1}-\left(\frac w{w_\sigma}\right)^{p-1}\right)\left((u_\sigma)^p-(w_\sigma)^p\right)\, \di \gamma_N=0,
\end{equation*}
and
\begin{equation*}
\lim_{\sigma\rightarrow0^+}\int_{u>w} \left(\frac w{w_\sigma}\right)^{p-1}\left((w_\sigma)^p-(u_\sigma)^p\right)\, \di \gamma_N=\int_{u>w} \left(w^p-u^p\right)\, \di \gamma_N\le0,
\end{equation*}
then
\begin{equation}\label{absa}
\liminf_{\sigma\rightarrow0^+}\int_{u>w} \left(\left(\frac 1{w_\sigma}\right)^{p-1}-\left(\frac 1{u_\sigma}\right)^{p-1}\right)\left((u_\sigma)^p-(w_\sigma)^p\right)\, \di \gamma_N
\le0,
\end{equation}
where the integrand is nonnegative, that is,
\begin{eqnarray*}
\left(\left(\frac 1{w_\sigma}\right)^{p-1}-\left(\frac 1{u_\sigma}\right)^{p-1}\right)\left((u_\sigma)^p-(w_\sigma)^p\right)
\ge0.
\end{eqnarray*}
If $u> w$ on a set of positive measure, by Fatou lemma we have
\begin{eqnarray*}
\liminf_{\sigma\rightarrow0^+}\int_{u>w} \left(\left(\frac 1{w_\sigma}\right)^{p-1}-\left(\frac 1{u_\sigma}\right)^{p-1}\right)\left((u_\sigma)^p-(w_\sigma)^p\right)\, \di \gamma_N\ge\\
\ge\int_{u>w} \left(\left(\frac 1{w}\right)^{p-1}-\left(\frac 1{u}\right)^{p-1}\right)\left(u^p-w^p\right)\, \di \gamma_N>0,
\end{eqnarray*}
which contradicts \eqref{absa}, and the claim is proved.
\vskip.2cm
\noindent \underline{Item (\ref{(ab)})}
\vskip.3cm

The monotonicity of $Q_p(\alpha, \Omega)$ with respect to $\alpha$ can be proven from the definition.
In order to prove the differentiation formula we firstly show that, for every $\alpha\in(-\infty, \lambda_1(\Omega))$, using the notation of item (\ref{(aa)}), we have
\begin{equation}\label{conve}
v_{(\alpha+\varepsilon)}\rightarrow v_{(\alpha)}, \quad \text{strongly in }L^p(\Omega,\phi_N),\text{  as }\varepsilon\rightarrow0.
\end{equation}
Let us put
$$
u_\varepsilon =v_{(\alpha+\varepsilon)}, \qquad u = v_{(\alpha)}.$$
Using the equation in \eqref{veq.0} satisfied by $u_\varepsilon$, we have, for $\delta>0$ suitably small and  for a suitable constant $C(\delta)>0$,
\begin{eqnarray*}
\int_\Omega|Du_\varepsilon|^p\, \di \gamma_N=(\alpha+\varepsilon)\int_\Omega |u_\varepsilon|^p\, \di \gamma_N+p\int_\Omega u_\varepsilon\, \di \gamma_N\\
\le\left(\frac{\alpha+\varepsilon}{\lambda_1(\Omega)}+\delta\right)\int_\Omega|Du_\varepsilon|^p\, \di \gamma_N+C(\delta)\gamma_N(\Omega).
\end{eqnarray*}
This means that $u_\varepsilon$ is bounded in $W^{1,p}(\Omega,\phi_N)$. Then there exists a subsequence $u_{\varepsilon_h}$ which strongly converges in $L^p(\Omega,\phi_N)$ for $\varepsilon$ which goes to 0. Actually, in view of the monotonicity proven in item
(\ref{(aa)}), we can say that the whole sequence $u_\varepsilon$ is such that
\begin{equation}\label{convu}
u_\varepsilon\rightarrow \bar u, \quad \text{strongly in }L^p(\Omega,\phi_N),\text{  as }\varepsilon\rightarrow0,
\end{equation}
for some $\bar u\in L^p(\Omega,\phi_N)$. A result contained, for example, in \cite{DF} (see also \cite{BM}) allows us to get the almost everywhere convergence of $Du_\varepsilon$, then we can pass to the limit as $\varepsilon \rightarrow0$ in the equation satisfied by $u_\varepsilon$. In view of the uniqueness stated in Proposition \ref{prop}, item (\ref{(b)}), we have that $\bar u=u$, then \eqref{conve} is proved.

In order to prove the differentiation formula, we observe that, using the definition of $Q_p(\alpha,\Omega)$ and of $F_\alpha(w)$ given in \eqref{QQ} and \eqref{F}, respectively, it holds:
\begin{equation*}
F_{\alpha+\varepsilon}(u_\varepsilon)-F_{\alpha}(u_\varepsilon)\ge Q_p(\alpha+\varepsilon,\Omega)-Q_p(\alpha,\Omega)
\ge F_{\alpha+\varepsilon}(u)-F_{\alpha}(u)
\end{equation*}
that is,
\begin{equation}\label{FQa}
\varepsilon\int_\Omega |u_\varepsilon|^p\, \di \gamma_N\ge Q_p(\alpha+\varepsilon,\Omega)-Q_p(\alpha,\Omega)
\ge \varepsilon\int_\Omega |u|^p\, \di \gamma_N.
\end{equation}
Taking into account \eqref{conve}, inequalities \eqref{FQa} imply
$$\lim_{\varepsilon\rightarrow0} \frac{Q_p(\alpha+\varepsilon,\Omega)-Q_p(\alpha,\Omega)}\varepsilon=
\int_\Omega |u|^p\, \di \gamma_N
$$
and the claim is proved.
\vskip.2cm
\noindent \underline{Item (\ref{(ac)})}
\vskip.3cm
Since \eqref{Q} holds true, we prove that 
\begin{equation}\label{limQ}
\displaystyle\lim_{\alpha\rightarrow-\infty} \int_\Omega v(x)\, \di\gamma_N=0  \,.
\end{equation}
To this aim we choose $v$ as test function in \eqref{veq.0} and we get
\begin{equation*}
\frac 1\alpha \int_\Omega |D v|^p\, \di\gamma_N\,=
\int_\Omega |v(x)|^p\, \di\gamma_N
+\frac 1\alpha \int_\Omega v(x)\, \di\gamma_N\,.
\end{equation*}
Since $\alpha<0$, by using H\"older inequality, we get
\begin{align*}
\int_\Omega | v|^p\, \di\gamma_N &\le-\frac 1\alpha\int_\Omega v(x)\, \di\gamma_N\\
\\
\displaystyle\quad&\le -\frac 1\alpha \gamma_N (\Omega)^{1-\frac 1p}
\left ( \int_\Omega |v(x)|^p\, \di\gamma_N \right)^\frac{1}{p}.
\end{align*}
We deduce
\begin{equation*}
\int_\Omega | v|^p\, \di\gamma_N \le \frac 1{(-\alpha)^{\frac p{p-1}}} \gamma_N (\Omega)
\end{equation*}
and therefore
\begin{equation}\label{limQ2}
\lim_{\alpha \to -\infty}\int_\Omega | v|^p\, \di\gamma_N =0\,.
\end{equation}
On the other hand, by using H\"older inequality, we get
\begin{equation}\label{limQ3}
\left |\int_\Omega  v\, \di\gamma_N \right |
\le 
 \gamma_N (\Omega)^{1-\frac 1p}
\left (\int_\Omega | v|^p\, \di\gamma_N\right)^\frac 1p.
\end{equation}
Combining \eqref{limQ}-\eqref{limQ3} the assert follows. 
\vskip.2cm
\noindent \underline{Item (\ref{(ad)})}
\vskip.3cm

We observe that, in view of the monotonicity of $Q_p(\alpha, \Omega)$ stated  item (\ref{(ab)}), the limit
\[\lim_{\alpha\rightarrow\lambda_1(\Omega)^-}Q_p(\alpha, \Omega)
\]
exists, finite or not. 
We can use as test function $w=ku$, where $k$ is an arbitrary positive constant and $u$ is a positive eigenfunction of problem \eqref{eq.0}, obtaining
\begin{align*}
\displaystyle Q_p(\alpha, \Omega)&\ge-\int_\Omega|D(ku)|^p\, \di \gamma_N+\alpha\int_\Omega |ku|^p\, \di \gamma_N+p\int_\Omega ku\, \di \gamma_N=\\
\\
\displaystyle\quad&= \bigl(\alpha-\lambda_1(\Omega)\bigr)k^p\int_\Omega |u|^p\, \di \gamma_N+pk\int_\Omega u\, \di \gamma_N.
\end{align*}
Letting $\alpha\rightarrow\lambda_1(\Omega)^-$, we have
\[\lim_{\alpha\rightarrow\lambda_1(\Omega)^-}Q_p(\alpha, \Omega)\ge pk\int_\Omega u\, \di \gamma_N
\]
and from the arbitrariness of $k$ the claim follows.
\end{proof}

When $\Omega$ is a half-space, all the results stated in Proposition \ref{prop} hold true, but some further properties about the behavior of $Q(\alpha, \Omega)$ with respect to the choice of the half-space can be added. So, we put, for $t\in\R$,
\begin{equation*}
H_t= \{x = (x_1, x_2, \dots , x_N) \in\RR : x_1 < t\}
\end{equation*}
and we introduce the function of two variables
\begin{equation}\label{Qrad}
Q_p^\sharp(\alpha,t)=Q_p(\alpha,H_{t}), \qquad\quad \alpha,\ t\in\R.
\end{equation}

{Let us observe that definition \eqref{Qrad} makes sense also when $t=+\infty$, that is, when one considers $\Omega=\R^N$. In such a case, it is well known that $\lambda_1(\R^N)=0$, and it is immediate to observe that
\begin{equation}\label{Qspace}
Q_p(\alpha, \R^N)=(p-1)\left(-\frac1\alpha\right)^{\frac1{p-1}}\qquad\text{for }\alpha<0,
\end{equation}
where maximum is achieved when $w=\left(-\frac1\alpha\right)^{\frac1{p-1}}$. Furthermore,
\begin{equation}\label{Qspacec}
Q_p^\sharp(\alpha,t)\le Q_p(\alpha, \R^N)\qquad\text{for }\alpha<0,\>t\in\R^N.
\end{equation}
Let us explicitly observe that the $p$-torsional rigidity defined in \eqref{TT}, that is, the case $\alpha=0$, is infinite when $\Omega=\R^N$.
}

\noindent We now prove some results about the behavior of $Q_p^\sharp(\alpha,t)$ with respect to $t$. We observe that for $p=2$ such properties could be proven using the explicit solution to \eqref{veq.0} in a half-space (see, e.g., Section \ref{s5}).

Before stating the results, we observe that, when one considers the generalized torsional rigidity on $H_t$, the use of Proposition \ref{prop} (\ref{(d)}) implies that the generalized torsion function $v$, solution to \eqref{veq.0} on $H_t$, depends only on the variable $x_1$, so we have
\begin{equation}\label{1dimQ}
Q_p^\sharp(\alpha,t)
=\max_{w\in W_{0}^{1,p} \left( (-\infty,t),\phi_{1} \right)}\left\{-\int_{-\infty}^t|w'(x)|^p\, \di \gamma_1+\alpha\int_{-\infty}^t |w(x)|^p\, \di \gamma_1+p\int_{-\infty}^t w(x)\, \di \gamma_1\right\}.
\end{equation}
and the function $v$, which achieves the maximum above, 
solves the problem
\begin{equation} \label{2eqQ}
\left\{
\begin{array}
[c]{lll}
-\bigl(|v'|^{p-2}v'\phi_{1}(\sigma)\bigr)'=\alpha v^{p-1}\phi_{1}(\sigma)+\phi_{1}(\sigma)& & \sigma\in
(-\infty,t),\\
\\
v\in W^{1,p}_{0}\!
\left(
(-\infty,t),\phi_{1}\right). & & 
\end{array}
\right.
\end{equation}
\begin{proposition}\label{finite}
Let $Q_p^\sharp(\alpha,t)$ be the function defined in \eqref{Qrad}, we have:
\vskip.2cm
\noindent- for any fixed $\alpha<0$, $Q_p^\sharp(\alpha,t)$ is finite for every $t\in\R$ and
\begin{equation}\label{itlim}
\lim_{t\rightarrow+\infty}Q_p^\sharp(\alpha,t)=(p-1)\left(-\frac1\alpha\right)^{\frac1{p-1}};
\end{equation}
\noindent- for $\alpha=0$, $Q_p^\sharp(0,t)$ is finite for every $t\in\R$ and
\begin{equation}\label{itlim0}
\lim_{t\rightarrow+\infty}Q_p^\sharp(0,t)=+\infty;
\end{equation}
\noindent- for any fixed $\alpha>0$, $Q_p^\sharp(\alpha,t)$  is finite if and only if
\begin{equation}\label{t}
t<{\bar t}
\end{equation}
where $\bar t\in\R$ is such that $\lambda_1(H_{\bar t})=\alpha$. Furthermore
\begin{equation}\label{tlim}
\lim_{t\rightarrow\,{\bar t^-}}Q_p^\sharp(\alpha,t)=+\infty.
\end{equation}
\end{proposition}
{
\begin{proof} Let us observe that, for all $\alpha$, $Q_p^\sharp(\alpha,t)$ is finite for the values of $t$ considered in the statement of the proposition, in view of Proposition \ref{prop} (\ref{(a)}).

Let us fix $\alpha<0$ and, for every $t\in\R^N$,  let us consider the following function defined on $H_t$
\begin{equation*}
w_t(x)=\left\{
\begin{array}
[c]{lll}%
\left(-\dfrac1\alpha\right)^{\frac1{p-1}}& & \text{if }x_1<t-1%
\\
\\
\left(-\dfrac1\alpha\right)^{\frac1{p-1}}(t-x_1)& & \text{if }t-1\le x_1<t%
\\
\\
0& & \text{if } x_1\ge t.
\end{array}
\right. %
\end{equation*}
If $F_\alpha(w)$ is defined as in \eqref{F} with $\Omega=H_t$,  in view of \eqref{Qspacec} and of the definition of $Q_p^\sharp(\alpha,t)$, we have:
\begin{equation}\label{wt0}
F_\alpha(w_t)\le Q_p^\sharp(\alpha,t)\le Q_p(\alpha,\R^N).
\end{equation}
A direct computation gives
$$F_\alpha(w_t)=-\left(-\frac1\alpha\right)^{\frac p{p-1}}\gamma _{N}({H_t\backslash H_{t-1}})+(p-1) \left(-\frac1\alpha\right)^{\frac 1{p-1}}\gamma _{N}({H_{t-1}})+\int_{H_t\backslash H_{t-1}}(\alpha w^p_t+pw_t)\,\di \gamma _{N}
$$
and it follows
\begin{equation*}
\lim_{t\rightarrow+\infty}F_\alpha(w_t)=(p-1)\left(-\frac1\alpha\right)^{\frac1{p-1}}.
\end{equation*}
Taking into account \eqref{Qspace} and \eqref{wt0}, we have \eqref{itlim}.

When $\alpha=0$ one can argue in a similar way. Consider, for $k>0$, the test function
\begin{equation*}
w_{t,k}(x)=\left\{
\begin{array}
[c]{lll}%
k& & \text{if }x_1<t-1%
\\
\\
k(t-x_1)& & \text{if }t-1\le x_1<t%
\\
\\
0& & \text{if } x_1\ge t.
\end{array}
\right. %
\end{equation*}
A direct computation gives
$$F_0(w_{t,k})\ge-k^{p}\gamma _{N}({H_t\backslash H_{t-1}})+(p-1) k\gamma _{N}({H_{t-1}})
$$
and, for every fixed $k>0$, it follows
\begin{equation*}
\liminf_{t\rightarrow+\infty}F_0(w_{t,k})\ge(p-1)k.
\end{equation*}
Being $k$ arbitrary, taking into account the inequality
\begin{equation*}
F_0(w_{t,k})\le Q_p^\sharp(0,t),
\end{equation*}
\eqref{itlim0}  follows. 

When $\alpha>0$, we observe that, adapting the arguments used in \cite{GMS} one can prove that $\lambda_1(H_{t})$ is continuous with respect to $t$, so, using the behaviour as $t\rightarrow-\infty$, one can say that for every $\alpha>0$ there exists $\bar t\in\R$ such that $\lambda_1(H_{\bar t})=\alpha$.

For $0<h<1$, we denote by $u_h$
the positive eigenfunction in $H_{\bar t-h}$ such that $\|u_h\|_{L^p(H_{\bar t-h},\phi_N)}=1$ and by $\bar u$
 the positive eigenfunction in $H_{\bar t}$ such that $\|\bar u\|_{L^p(H_{\bar t},\phi_N)}=1$.
 Using the equation satisfied by $u_h$ we have that $\|u_h\|_{W_0^{1,p}(H_{\bar t-h},\phi_N)}$ is bounded. Letting $u_h(x)$ be extended to 0 for $x_1>\bar t-h$, we obtain that 
there exists a subsequence $u_{h'}$ which weakly converges in $W_0^{1,p}(H_{\bar t},\phi_N)$ to a function $v\in W_0^{1,p}(H_{\bar t},\phi_N)$. It follows that there exists a further subsequence, still denoted by $u_{h'}$ which converges to $v$ strongly in $L^{p}(H_{\bar t},\phi_N)$ and almost everywhere. On the other hand, we have already observed that $u_h$ depends only on the first variable and it solves problem \eqref{2eq}, that is
$$-|u_{h'}'(s)|^{p-2}u_{h'}'(s)\phi_{1}(s)=\lambda_1(H_{\bar t-{h'}})\int_{-\infty}^su_{h'}^{p-1}\phi_{1}(\sigma)\,\di \sigma,\qquad s\in
(-\infty,\bar t-h).
$$
It is possible to pass to the limit in the above equality obtaining that $u_{h'}'$ converges almost everywhere in $H_{\bar t}$, and then strongly in ${L^q(H_{\bar t},\phi_N)}$, $1\le q<p$, to $v'$. Then $v$ satisfies \eqref{2eq} in the interval $(-\infty, \bar t)$ and, in view of the simplicity of the first eigenvalue, we have $v=\bar u$. We finally observe that, being the limit unique, all the sequence $u_h$ converges to $\bar u$ in the ways we have outlined above.

In order to evaluate the limit in \eqref{tlim} we use $w=k  u_h$ as a test function, where $k$ is an arbitrary positive constant, to obtain
\begin{align*}
\displaystyle Q_p^\sharp(\alpha, \bar t-h)&\ge-\int_{H_{\bar t-h}} |Dk  u_h|^p\, \di \gamma_N+\alpha\int_{H_{\bar t-h}} |k u_h|^p\, \di \gamma_N+p\int_{H_{\bar t-h}}  k u_h\, \di \gamma_N\\
\\
\displaystyle\quad&= \bigl(\alpha-\lambda_1(H_{\bar t-h})\bigr)k^p\int_{H_{\bar t-h}}  |u_h|^p\, \di \gamma_N+pk\int_{H_{\bar t-h}}  u_h\, \di \gamma_N
\end{align*}
Letting $h\rightarrow0^+$, we have
\[
\lim_{h\rightarrow0^+}Q_p^\sharp(\alpha, \bar t-h)\ge pk\int_{H_{\bar t}}\bar u\, \di \gamma_N
\]
and from the arbitrariness of $k$ the claim follows.

\end{proof}
}
\begin{proposition}\label{derQ}
For any fixed $\alpha$, the function $Q_p^\sharp(\alpha,t)$ defined in \eqref{Qrad} is strictly increasing with respect to $t$ and the following differentiation formula holds true
\begin{equation}\label{shapeQ}
\frac\partial{\partial t}Q_p^\sharp(\alpha,t)=(p-1)|v'(t)|^p\phi_1(t).
\end{equation}
where $v$ is the solution to \eqref{2eqQ}.
\end{proposition}
Let us observe that, as for Proposition \ref{half}, formula \eqref{shapeQ} can be seen as a shape derivative of $Q_p(\alpha,\Omega)$ on half-spaces. Its proof uses arguments similar to those used to prove Proposition \ref{half}, so it can be found in
Section \ref{s6}.

An immediate consequence of the above results is the following proposition.
\begin{proposition}\label{prad}
For any $\Omega\subset \R^N$, with $\gamma_N (\Omega)<1$, and for all $-\infty<\alpha<\lambda_1(\Omega)$, there exists a unique $t(\alpha)\in\R$, with $H_{t(\alpha)}\subset\Omega^\sharp$, such that
\begin{equation}\label{cnd}
Q_p(\alpha,\Omega)= Q_p(\alpha,H_{t(\alpha)}).
\end{equation}
Furthermore, if $\alpha>0$, we have $t(\alpha)<\bar t$, where $\bar t$ is such that $\lambda_1(H_{\bar t})=\alpha$.
\end{proposition}
\begin{proof}
In view of Proposition \ref{derQ} and \eqref{tlim}, for a fixed $\alpha >0$, the strictly increasing function with respect to the variable $t$, $Q_p^\sharp(\alpha,t)$ maps $(-\infty,\bar t)$ into $(0,+\infty)$, so there exists a unique $t(\alpha)<\bar t$ such that \eqref{cnd} holds.
When $\alpha=0$, we use \eqref{itlim0} and we repeat the same argument. Obviously, in this case we have only $t(0)<+\infty$.

When $\alpha<0$, we firstly observe that, in view of Proposition \ref{prop} (\ref{(c)}) and \eqref{Qspace}, it holds
\begin{equation*}
Q_p(\alpha, \Omega)<Q_p(\alpha, \R^N)=(p-1)\left(-\frac1\alpha\right)^{\frac1{p-1}}.
\end{equation*}
On the other hand, \eqref{itlim} implies that the strictly increasing function $Q_p^\sharp(0,t)$ maps $\R$ into $(0,Q_p(\alpha, \R^N))$, so there exists a unique $t(\alpha)\in\R$, with $H_{t(\alpha)}\subset\Omega^\sharp$, such that \eqref{cnd} holds.

We finally observe that the inclusion $H_{t(\alpha)}\subset\Omega^\sharp$ follows from the fact that $Q_p(\alpha,H_{t(\alpha)})$ is finite.
\end{proof}
\section{Comparison results}\label{s3}

In this section we show that one can use standard symmetrization arguments which go back to Talenti results \cite{T} (see, for example \cite{BBMP} in the gaussian context) in order to prove a comparison result which allows us to estimate the generalized torsion function introduced in Section \ref{s2}. Let us observe that, in the unweighted case, similar results have been obtained, for example, in \cite{C1}, \cite{C2} when $p=2$, or in \cite{AFT} when $p>1$.

\begin{theorem}\label{comp}
For a fixed $\alpha\in(-\infty,\lambda_1(\Omega))$, let $v$ be the solution to problem \eqref{veq.0} and let $t\in\R$ be such that
$Q_p(\alpha,\Omega)= Q_p^\sharp(\alpha,t)$. If $\bar v$ is the solution to problem
\begin{equation} \label{veq}
\left\{
\begin{array}
[c]{lll}%
-\dive(\phi_N (x)|D\bar v|^{p-2}D\bar v)=\alpha \phi_N (x)|\bar v|^{p-2}\bar v+\phi_N (x)& & \text{in }%
H_t,\\
\\
\bar v=0 & & \text{on }\partial H_t,
\end{array}
\right. %
\end{equation}
then the following assertions hold:
\begin{itemize}
\item if $\alpha\ge0$, then, for every $1\le m<\infty$ and for every $r\le t$,
\begin{equation}\label{qq}
\int_{H_r}\bigl(v^\sharp(x)\bigr)^m\, \di \gamma_N \le \int_{H_r} \bigl(\bar v(x)\bigr)^m\, \di \gamma_N;
\end{equation}
\item if $\alpha<0$, then \eqref{qq} holds for every $\max\{1,p-1\}\le m<\infty$ and for every $r\le t$.
\end{itemize}
\end{theorem}

\begin{proof}
By standard arguments, see for example \cite{BBMP}, \cite{DF}, we can verify that for $v^*$ the following inequality holds true
\begin{equation}\label{uprob}
\frac 1{(2\pi)^{\frac p2}} \exp\left (-\frac{p(k^{-1}(s))^2}{2}\right)
\bigl(-(v^*)'(s)\bigr)^{p-1}
\le \alpha \int_0^s\bigl(v^*(\sigma)\bigr)^{p-1} \,\di\sigma +s\,, \quad\hbox{ a.e. in }
(0,\gamma_N(\Omega)),
\end{equation}
while $\bar v^*$ satisfies
\begin{equation}\label{vprob}
\frac 1{(2\pi)^{\frac p2}} \exp\left (-\frac{p(k^{-1}(s))^2}{2}\right)
\bigl(-(\bar v^*)'(s)\bigr)^{p-1}
= \alpha \int_0^s\bigl(\bar v^*(\sigma)\bigr)^{p-1} \,\di\sigma +s\,, \quad\hbox{ a.e. in }
(0, \gamma_N(H_t)),
\end{equation}
and we will understand that $v^*(s)$ and $\bar v^*(s)$ are extended to 0, respectively, for 
$s>\gamma_N(\Omega)$ and for
$s>\gamma_N(H_t)$.

Let us observe that, in view of Proposition \ref{prop}, items (\ref{(c)}), (\ref{(d)}), it holds
\begin{equation}\label{compH}
\gamma_N (H_t) \le 
\gamma_N (\Omega).
\end{equation}
Furthermore, the equality $Q_p(\alpha,\Omega)= Q_p^\sharp(\alpha,t)$ implies
\begin{equation}\label{obs}
\int_0^{\gamma_N (\Omega)}v^*(\sigma)\,\di\sigma=\int_0^{\gamma_N (H_t)}\bar v^*(\sigma) \,\di\sigma.
\end{equation}
Then it is well defined
\[
\bar s=\sup\{s>0:v^*(s)=\bar v^*(s)\}\le\gamma_N (H_t).
\]
We first assume that $\alpha\ge0$. Let us put
\begin{equation}\label{v}
w(s)=\left\{
   \begin{array}{ll}
\max\{v^*(s),\bar v^*(s)\},&0< s<\bar s,\\
&\\
\bar v^*(s),&\text{otherwise}.
\end{array}
 \right.
\end{equation}
Then from \eqref{uprob} and \eqref{vprob} we know that for a.e. $s\in (0,\gamma_N (H_t))$,
\begin{equation}\label{wprob}
 \left\{
   \begin{array}{lll}\displaystyle
\frac 1{(2\pi)^{\frac p2}} \exp\left (-\frac{p(k^{-1}(s))^2}{2}\right)
\bigl(-w'(s)\bigr)^{p-1}
\le \alpha \int_0^s\bigl(w(\sigma)\bigr)^{p-1} \,\di\sigma +s\,,\\
&&\\
w(s)\ge\bar v^*(s)\,.
\end{array}
 \right.
\end{equation}
Multiplying the first inequality by $-w'(s)$ and integrating, we have
\begin{equation}\label{intW}
\int_0^{\gamma_N ( H_t)}|w'(s)|^{p}
\frac 1{(2\pi)^{\frac p2}} \exp\left (-\frac{p(k^{-1}(s))^2}{2}\right)
\,\di s
\le
\alpha\int_0^{\gamma_N ( H_t)} |w(s)|^p\,\di s +\int_0^{\gamma_N ( H_t)}w(s)\,\di s\,.
\end{equation}
We can also use $\tilde w(x)=w(k(x_1))$ as a test function in $Q^\sharp_{p}(\alpha,t)$, to get
\begin{align*}
Q^\sharp_{p}(\alpha,t)
&\ge-\int_{-\infty}^{t}|\tilde w'(x_1)|^{p}
\phi_1(x_1)
\,\di x_1+\alpha
\int_{-\infty}^{t} |\tilde w(x_1)|^{p}
\phi_1(x_1)
\,\di x_1 +p\int_{-\infty}^{t}\tilde w(x_1)\phi_1(x_1)
\,\di x_1
\\
&=-\int_0^{\gamma_N ( H_t)}|w'(s)|^{p}
\frac 1{(2\pi)^{\frac p2}} \exp\left (-\frac{p(k^{-1}(s))^2}{2}\right)
\,\di s+\alpha
\int_0^{\gamma_N ( H_t)} |w(s)|^p\,\di s +p\int_0^{\gamma_N ( H_t)}w(s)\,\di s
\\
&\ge (p-1)\int_0^{\gamma_N ( H_t)}w(s)\,\di s\ge
(p-1)\int_0^{\gamma_N ( H_t)}\bar v^*(s)\,\di s
=
Q^\sharp_{p}(\alpha,t).
\end{align*}
By the characterization of the maximum we have $w=\bar v^*$, that is
\[
\left\{
   \begin{array}{lll}\displaystyle
\bar v^*(s)\ge v^*(s) \quad\text{in }(0,\bar s),\\
&&\\
v^*(s)\ge \bar v^*(s) \quad\text{otherwise}.
\end{array}
 \right.
\]
Using again \eqref{obs}, we have
\[
\int_0^{s} v^*(\sigma)\,\di \sigma\le\int_0^{s} \bar v^*(\sigma)\,\di \sigma,
\qquad s>0.
\]
Hence \eqref{qq} follows by Proposition \ref{relazione}, for every $1\le m<\infty$.

Let us now assume that $\alpha < 0$. 

We firstly observe that, under this condition, both $v$ and $\bar{v}$ are bounded. 
Indeed, let us define 
\[
M_\alpha = \left( - \frac{1}{\alpha} \right)^{\frac{1}{p-1}}.
\]
By choosing $(v - M_\alpha)_+$ as a test function in \eqref{veq.0}, we obtain
\[
\int_{\{v > M_\alpha\}} |Dv|^p \, d\gamma_N = \int_{\{v > M_\alpha\}} \left( \alpha v^{p-1} + 1 \right) (v - M_\alpha) \, d\gamma_N \le 0,
\]
since $\alpha v^{p-1} + 1 \le 0$ on the set where $v > M_\alpha$. This implies that $(v - M_\alpha)_+ = 0$ a.e. in $\Omega$, which leads to
\[
0 \le v \le M_\alpha \quad \text{a.e. in } \Omega.
\]
A similar argument applied to $\bar{v}$ yields
\[
0 \le \bar{v} \le M_\alpha \quad \text{a.e. in } H_t.
\]
In particular, we conclude that
\begin{equation} \label{M_a}
v^*(0) = \operatorname*{ess\,sup}_{\Omega} v \in (0, M_\alpha], \qquad 
\bar{v}^*(0) = \operatorname*{ess\,sup}_{H_t} \bar{v} \in (0, M_\alpha].
\end{equation}

We then consider the function
\[
\Psi(s)=\int_0^{s}\left(\bigl(v^*(\sigma)\bigr)^{p-1}-\bigl(\bar v^*(\sigma)\bigr)^{p-1}\right)\, \di \sigma,
\qquad0\le s\le\gamma_N(\Omega).
\]

We claim that $\Psi$ cannot achieve a local positive maximum in $(0,\gamma_N(\Omega))$. Indeed, the interval
$\gamma_N(H_t)\le s\le\gamma_N(\Omega)$ is immediately excluded, since there $\bar v^*(s)=0$ and therefore $\Psi$ is nondecreasing, in fact increasing unless $v^*=0$.

Suppose, by contradiction, that $\Psi$ has a local positive maximum at some point $s_0\in(0,\gamma_N(H_t))$.
Then  for some $\delta>0$, it holds
\begin{equation}\label{aneg}
\Psi(s_0)=\max_{|s-s_0|<\delta}\Psi(s)>0,
\end{equation}
with $v^*(s_0)-\bar v^*(s_0)=0$.
This means that the function $v^*(s)- \bar v^*(s)$ is strictly increasing, with $v^*(s_0)-\bar v^*(s_0)=0$, then \eqref{aneg} cannot hold and
the claim follows.

We finally distinguish two possibilities.

If
\[
\Psi(s)\le0\qquad\forall s\in(0,\gamma_N(\Omega)),
\]
then
\[
\int_0^{s}\bigl(v^*(\sigma)\bigr)^{p-1}\,\di \sigma
\le
\int_0^{s}\bigl(\bar v^*(\sigma)\bigr)^{p-1}\,\di \sigma,
\qquad s>0.
\]
Therefore \eqref{qq} follows by Proposition \ref{relazione} for every $m\ge p-1$. In particular, this gives the desired conclusion when $p\ge2$.

It remains to consider the case $1<p\le2$. In this case $p-1\le1$. If $\Psi(s)>0$ for some $s$, the previous claim implies that the positive maximum of $\Psi$ is attained at $s=\gamma_N(\Omega)$. Hence there exists $\tilde s\in(0,\gamma_N(\Omega))$ such that
\[
\left\{
   \begin{array}{ll}\displaystyle
\Psi(s)\le0 &\quad\text{for }0\le s\le\tilde s,\\
&\\
\Psi(s)>0 &\quad\text{for }\tilde s< s\le\gamma_N(\Omega),\\
&\\
v^*(s)\ge \bar v^*(s) &\quad\text{for }\tilde s< s\le\gamma_N(\Omega).
\end{array}
 \right.
\]
Together with \eqref{obs}, this yields
\[
\int_0^{s}v^*(\sigma)\,\di\sigma
\le
\int_0^{s}\bar v^*(\sigma)\,\di\sigma,
\qquad s>0.
\]
Thus, by Proposition \ref{relazione}, \eqref{qq} follows for every $1\le m<\infty$.

Finally, let $p>2$. We show that the alternative
\[
\max_{0\le s\le\gamma_N(\Omega)}\Psi(s)=\Psi(\gamma_N(\Omega))>0
\]
cannot occur. For $s<\gamma_N(\Omega)$, using the inequality
\[
a-b\ge \frac{a^{2-p}}{p-1}\left(a^{p-1}-b^{p-1}\right),
\qquad a,b>0,
\]
we get
\begin{align*}
\int_0^{s}\left(v^*(\sigma)-\bar v^*(\sigma)\right)\, d\sigma
&\ge
\int_0^{s}\left(\bigl(v^*(\sigma)\bigr)^{p-1}-\bigl(\bar v^*(\sigma)\bigr)^{p-1}\right)
\frac{\bigl(v^*(\sigma)\bigr)^{2-p}}{p-1}\, d\sigma
\\
&=
\frac{\bigl(v^*(s)\bigr)^{2-p}}{p-1}\Psi(s)
-\int_0^{s}
\left(\frac{\bigl(v^*(\sigma)\bigr)^{2-p}}{p-1}\right)'\Psi(\sigma)\, d\sigma
\\
&\ge
\frac{\bigl(v^*(0)\bigr)^{2-p}}{p-1}\Psi(\gamma_N(\Omega)) \\
&\quad -\frac{\bigl(v^*(s)\bigr)^{2-p}}{p-1}
\int_s^{\gamma_N(\Omega)}
\left(\bigl(v^*(\sigma)\bigr)^{p-1}-\bigl(\bar v^*(\sigma)\bigr)^{p-1}\right)\, d\sigma.
\end{align*}
The last term vanishes as $s\rightarrow\gamma_N(\Omega)$, because
\[
\frac1{\bigl(v^*(s)\bigr)^{p-2}}
\int_s^{\gamma_N(\Omega)}
\left(\bigl(v^*(\sigma)\bigr)^{p-1}-\bigl(\bar v^*(\sigma)\bigr)^{p-1}\right)\,\di\sigma
\le
v^*(s)\bigl(\gamma_N(\Omega)-s\bigr).
\]
Passing to the limit as $s\rightarrow\gamma_N(\Omega)$, taking into account of  \eqref{M_a}, we obtain
\[
\int_0^{\gamma_N(\Omega)}
\left(v^*(\sigma)-\bar v^*(\sigma)\right)\,\di \sigma
\ge
\frac{\bigl(v^*(0)\bigr)^{2-p}}{p-1}\Psi(\gamma_N(\Omega))>0,
\]
which contradicts \eqref{obs}. Therefore
\[
\Psi(s)\le0\qquad\forall s\in(0,\gamma_N(\Omega)).
\]
Consequently,
\[
\int_0^{s}\bigl(v^*(\sigma)\bigr)^{p-1}\,\di\sigma
\le
\int_0^{s}\bigl(\bar v^*(\sigma)\bigr)^{p-1}\,\di\sigma,
\qquad s>0.
\]
By Proposition \ref{relazione}, applied to $\bigl(v^*\bigr)^{p-1}$ and $\bigl(\bar v^*\bigr)^{p-1}$ with the convex function
$F(\tau)=\tau^{m/(p-1)}$, we obtain \eqref{qq} for every $m\ge p-1$.

The proof is complete.
\end{proof}

We conclude the section observing that the arguments used above allow us to get a comparison result also for the first eigenfunction for problem \eqref{eq.0} (see \cite{BCF}) for the case $p=2$).
\begin{theorem}\label{compPR}
Let $u$ be a positive eigenfunction for problem \eqref{eq.0} and let $t\in\R$ be such that
$\lambda_1(\Omega)= \lambda_1(H_t)$. If, for $q\ge1$, $w_q$ is the first eigenfunction for problem
\begin{equation} \label{veqP}
\left\{
\begin{array}
[c]{lll}%
-\dive(\phi_N (x)|Dw_q|^{p-2}D\bar v)=\lambda_1(\Omega) \phi_N (x)|w_q|^{p-2}w_q& & \text{in }%
H_t,\\
\\
w_q=0 & & \text{on }\partial H_t,
\end{array}
\right. %
\end{equation}
such that
\begin{equation}\label{qqq}
\int_{\Omega}\bigl(u(x)\bigr)^q\, \di \gamma_N = \int_{H_t} \bigl(w_q(x)\bigr)^q\, \di \gamma_N.
\end{equation}
then, if $q\le m<\infty$, for every $r\le t$ it holds
\begin{equation}\label{qqP}
\int_{H_r}\bigl(u^\sharp(x)\bigr)^m\, \di \gamma_N \le \int_{H_r} \bigl(w_q(x)\bigr)^m\, \di \gamma_N.
\end{equation}
\end{theorem}

\begin{proof}
By standard arguments (see, for example, \cite{BBMP}, \cite{DF}) we can verify that for $v^*$ the following inequality holds true
\begin{equation}\label{upr}
-
\frac 1{(2\pi)^{\frac p2}} \exp\left (-\frac{p(k^{-1}(s))^2}{2}\right)\bigl((u^*)'(s)\bigr)^{p-1}
\le \lambda_1(\Omega) \int_0^s\bigl(u^*(\sigma)\bigr)^{p-1} \,\di\sigma \,, \quad\hbox{ a.e. in }
(0,\gamma_N(\Omega)),
\end{equation}
while $\bar v^*$ satisfies
\begin{equation}\label{vpr}
-\frac 1{(2\pi)^{\frac p2}} \exp\left (-\frac{p(k^{-1}(s))^2}{2}\right)\bigl((\bar w_q^*)'(s)\bigr)^{p-1}
= \lambda_1(\Omega) \int_0^s\bigl(w_q^*(\sigma)\bigr)^{p-1} \,\di\sigma \,, \quad\hbox{ a.e. in }
(0, \gamma_N(H_t)),
\end{equation}
and we will understand that $u^*(s)$ and $w_q^*(s)$ are extended to 0, respectively, for 
$s>\gamma_N(\Omega)$ and for
$s>\gamma_N(H_t)$.

Let us observe that, in view of Theorem \ref{KR}, it holds
\begin{equation}\label{comH}
\gamma_N (H_t) \le 
\gamma_N (\Omega),
\end{equation}
then, taking into account \eqref{qqq}, it is well defined
$$\bar s=\sup\{s>0:u^*(s)=\bar w_q^*(s)\}\le\gamma_N (H_t).
$$
Let us put
\begin{equation}\label{vP}
z(s)=\left\{
   \begin{array}{ll}
\max\{u^*(s), w_q^*(s)\},&0< s<\bar s\\
&\\
w_q^*(s),&\text{otherwise}
\end{array},
 \right.
\end{equation}
then from \eqref{upr} and \eqref{vpr} we know that for a.e. $s\in (0,\gamma_N (H_t))$,
\begin{equation}\label{zprob}
 \left\{
   \begin{array}{lll}\displaystyle
-
\frac 1{(2\pi)^{\frac p2}} \exp\left (-\frac{p(k^{-1}(s))^2}{2}\right)\bigl(z'(s)\bigr)^{p-1}
\le \lambda_1(\Omega) \int_0^s\bigl(z(\sigma)\bigr)^{p-1} \,\di\sigma \,,\\
&&\\
z(s)\ge w_q^*(s)\,.
\end{array}
 \right.
\end{equation}
Multiplying the first inequality by $-z'(s)$ and integrating, we have
\begin{equation}\label{intz}
\int_0^{\gamma_N ( H_t)}|z'(s)|^{p}
\frac 1{(2\pi)^{\frac p2}} \exp\left (-\frac{p(k^{-1}(s))^2}{2}\right)
\,\di s
\le
\lambda_1(\Omega)\int_0^{\gamma_N ( H_t)} |z(s)|^p\,\di s \,.
\end{equation}
We can also use $\tilde z(x)=z(k(x))$ as a test function in $\lambda_1(H_t)=\lambda_1(\Omega)$, to get
\begin{equation*}
\lambda_1(H_t)\le \dfrac{\displaystyle\int_{-\infty}^{t}|\tilde z'(\sigma)|^{p}\phi_1(\sigma)\, \di\sigma}{\displaystyle\int_{-\infty}^{t} |\tilde z(\sigma)|^p \phi_1(\sigma)\,\di \sigma}=\dfrac{\displaystyle\int_0^{\gamma_N ( H_t)}|z'(s)|^{p}
\frac 1{(2\pi)^{\frac p2}} \exp\left (-\frac{p(k^{-1}(s))^2}{2}\right)
\,\di s}{\displaystyle\int_0^{\gamma_N ( H_t)} |z(s)|^p\,\di s}
\le
\lambda_1(H_t).
\end{equation*}
By the characterization of the first eigenvalue we have $z=w_q^*$, that is,
\begin{equation*}
\left\{
   \begin{array}{lll}\displaystyle
w_q^*(s)\ge u^*(s) \quad\text{in }(0,\bar s),\\
&&\\
u^*(s)\ge w_q^*(s) \quad\text{otherwise}\,.
\end{array}
 \right.
\end{equation*}
Using again \eqref{qqq}, we have
\begin{equation*}
\int_0^{s} \bigl(u^*(\sigma)\bigr)^q\,\di \sigma\le\int_0^{s} \bigl(w_q^*(\sigma)\bigr)^q\,\di \sigma,\qquad s>0,
\end{equation*}
and, finally,  \eqref{qqP} follows
by Proposition \ref{relazione}.
\end{proof}
An immediate consequence of the above comparison result is the following Payne-Rayner
type inequality.
\begin{theorem}\label{PR}
Let $\lambda_1(\Omega)$ the first eigenfunction for problem \eqref{eq.0} and let $u$ be any eigenfunction associated to it. Then, for $1\le q\le m<\infty$, we have
\begin{equation*} 
\|u\|_{L^m(\Omega,\phi_N)} \le \beta \|u\|_{L^q(\Omega,\phi_N)},
\end{equation*}
where
\begin{equation*} 
\beta =\dfrac{\|w_q\|_{L^m(\Omega,\phi_N)}}{\|w_q\|_{L^q(\Omega,\phi_N)}},
\end{equation*}
with $w_q$ defined as in Theorem \ref{compPR}.
\end{theorem}

\section{Main result}\label{s4}

In this section we prove the main result of the paper which can be stated as follows.

\begin{theorem}\label{final}
For any $\alpha\in(-\infty,\lambda_1(\Omega))$ and for any set $\Omega\subset\RR$ with $\gamma_N (\Omega)<1$, letting $H_{t(\alpha)}\subset\Omega^\sharp$ be the half-space such
that $Q_p(\alpha,\Omega)= Q_p(\alpha,H_{t(\alpha)})$, we have
\begin{equation}\label{mainre}
\lambda_1(\Omega)\ge \lambda_1(H_{t(\alpha)}).
\end{equation}
\end{theorem}

\begin{remark} If in Theorem \ref{final} we put $p=2$ and $\alpha =0$, we obtain the inequality proved in \cite{HL}, which is the gaussian version of the inequality proved by Kohler-Jobin in the case of laplacian in the euclidean setting.
\end{remark}

Before giving the proof of Theorem \ref{final}, we prove the following result.
\begin{proposition}\label{monot}
For any $\alpha\in(-\infty,\lambda_1(\Omega))$, let $t(\alpha)\in\R$ be such that
$Q_p(\alpha,\Omega)= Q_p(\alpha,H_{t(\alpha)})$. The function $t(\alpha)$ is decreasing.
\end{proposition}

\begin{proof}
Using notation \eqref{Qrad}, $t(\alpha)$ is implicitly defined by the equality
\[Q_p^\sharp(\alpha,t)-Q_p(\alpha,\Omega)= 0.\]
Collecting Proposition \ref{propa} (\ref{(ab)}), Proposition \ref{prad} and Proposition \ref{derQ} we have that $t(\alpha)$ is well defined and 
\[t'(\alpha)=-\frac{\frac\partial{\partial\alpha}(Q_p^\sharp(\alpha,t)-Q_p(\alpha,\Omega))}{\frac\partial{\partial t}Q_p^\sharp(\alpha,t)},
\]
where $\frac\partial{\partial t}Q_p^\sharp(\alpha,t)>0$ and
\[
\frac\partial{\partial t}\bigl(Q_p^\sharp(\alpha,t)-Q_p(\alpha,\Omega)\bigr)=\int_{H_{t(\alpha)}} |\bar v(x)|^p\, \di \gamma_N-
\int_\Omega |v(x)|^p\, \di \gamma_N,
\]
where $v$ is the solution to \eqref{veq.0} and $\bar v$ be the solution to \eqref{veq.0}.
By Theorem \ref{comp} we have
$$
t'(\alpha)\le 0
$$
and the claim is proven.
\end{proof}

\begin{proof}[Proof of Theorem \ref{final}]
Let us observe that Proposition \ref{prad} implies that the half-space $H_{t(\alpha)}\subset\Omega^\sharp$ is well defined.
We observe that $Q_p(\alpha,H_{t(\alpha)})$ is finite from Proposition \ref{prop} (\ref{(a)}).
Furthermore,  Proposition \ref{prad} states also that for any $\alpha>0$ it holds that
\[t(\alpha)<\bar t,
\]
where $\bar t\in \R$ is such that $\lambda_1(H_{\bar t})=\alpha$.
Hence, the monotonicity of $t(\alpha)$, proven in Proposition \ref{monot}, implies that the following limit exists
\[\lim_{\alpha\rightarrow\lambda_1(\Omega)^{-}}t(\alpha)=\ell.
\]
We claim that $\ell\ge\tilde t$, where $\tilde t\in \R$ is such that $\lambda_1(H_{\tilde t})=\lambda_1(\Omega)$.
If, by contradiction,
\[
\ell<\tilde t ,
\]
taking into account the fact that $\lambda_1(H_{\ell})>\lambda_1(H_{\tilde t})=\lambda_1(\Omega)$ it would follow
\[\lim_{\alpha\rightarrow \lambda_1(\Omega)^{-}} Q_p(\alpha,\Omega)=\lim_{\alpha\rightarrow \lambda_1(\Omega)^{-}} Q_p^\sharp(\alpha,t(\alpha))=Q_p^\sharp(\lambda_1(\Omega),\ell)<+\infty,
\]
in contrast with Proposition \ref{propa} (\ref{(ad)}). Then
\[\lim_{\alpha\rightarrow\lambda_1(\Omega)^{-}}t(\alpha)\ge\tilde t
\]
and the monotonicity of $t(\alpha)$ gives $t(\alpha)\ge \tilde t$. Finally, being the first eigenvalue decreasing with respect to the inclusion of sets, we get
$$\lambda_1(\Omega)=\lambda_1(H_{\tilde t})\ge \lambda_1(H_{t(\alpha)}).$$
\end{proof}

\section{An explicit solution when \texorpdfstring{\(p=2\)}{p=2}}\label{s5}

Let us consider the case $p=2$. We observe the solution $v$ to problem \eqref{veq.0} when $\Omega=H_t$ depends only on the first variable, so it reduces to solve the following one-dimensional problem
\begin{equation} \label{1deq}
\left\{
\begin{array}
[c]{lll}%
-v''+xv'=\alpha v+1& & \text{in }%
(-\infty,t),\\
\\
v\in H_0^1((-\infty,t),\phi_1). & & 
\end{array}
\right. %
\end{equation}
Here we want to show that, for example, when $\alpha<0$, an explicit solution to \eqref{1deq} can be given. So we consider the problem
\begin{equation}\label{1eq}
\left\{
\begin{array}
[c]{lll}
-w''+xw'+aw=1& & \text{in }
(-\infty,t),\\
\\
w\in H_0^1((-\infty,t),\phi_1) & & 
\end{array}
\right. 
\end{equation}
where $a>0$ and $t\in\R$.

We first consider the homogeneous equation
\begin{equation}\label{hom}
-w''+xw'+aw=0
\end{equation}
and we observe that two independent solution to \eqref{hom} are given by
\begin{equation} \label{homsol}
\begin{array}
[c]{lll}
w_1=M\left(\frac a2,\frac12,\frac{x^2}2\right)& &\\
\\
w_2=xM\left(\frac {a+1}2,\frac32,\frac{x^2}2\right)& &
\end{array}
\end{equation}
where $M(b,c,x)$ is known as Kummer function and solves the Kummer equation:
\begin{equation}\label{Kum}
xv''+(c-x)v'-bv=0.
\end{equation}
The function $M(b,c,x)$ is also denoted as $_1F_1(b,c,x)$ and is defined as a function of a complex variable $z$ in the following way (we refer to \cite{AS}):
\begin{equation*}
M(b,c,z)=\sum_{k=0}^\infty \frac{(c)_k}{(b)_k}\frac{z^k}{k!}
\end{equation*}
where, for $\zeta\in\C$, $(\zeta)_k$ denotes the Pochhammer symbol
\begin{equation*}
(\zeta)_k=\frac{\Gamma(\zeta+k)}{\Gamma(\zeta)}, \qquad k\ge0,\> k\in \Ze.
\end{equation*}
It is known that 
\begin{equation}\label{asympm}
M(b,c,z)\approx\frac{\Gamma(c)}{\Gamma(b)}e^zz^{b-c}, \qquad\text{as }z\rightarrow\infty,\qquad \text{Re}(z)>0,
\end{equation}
while, for $c\not\in\Ze$ the following combination
\begin{equation}\label{U}
U(b,c,z)=\frac{\pi}{\sin\pi c}\left(\frac{M(c-b,c,z)}{\Gamma(1+b-c)\Gamma(c)}-\frac{M(1-b,2-c,z)}{\Gamma(b)\Gamma(2-c)}z^{1-c}\right)
\end{equation}
satisfies
\begin{equation}\label{asympu}
U(b,c,z)\approx z^{-b}, \qquad\text{as } \text{Re}(z)\rightarrow+\infty.
\end{equation}
Combining the above information and denoting
\begin{equation}\label{Y}
Y(a,x)=\frac{M\left(\frac a2,\frac12,\frac{x^2}2\right)}{\Gamma(1+b-c)\Gamma(c)}+\frac{M\left(\frac {a+1}2,\frac32,\frac{x^2}2\right)}{\Gamma(b)\Gamma(2-c)}\frac x{\sqrt2}.
\end{equation}
it turns out that the function
\begin{equation}\label{solw}
w(x)=\frac1a-\frac1a\frac{Y(a,x)}{Y(a,t)}
\end{equation}
is the solution to problem \eqref{1eq}, then
\begin{equation}\label{solv}
w(x)=\frac1\alpha\frac{Y(-\alpha,x)}{Y(-\alpha,t)}-\frac1\alpha
\end{equation}
solves problem \eqref{1deq}.

\section{Appendix}\label{s6}

Our aim is to prove Proposition \ref{half}, so we firstly give the
following regularity result.

\begin{proposition}\label{summab}
Let $u>0$ be a solution to the problem 
\begin{equation} \label{3eq}
\left\{
\begin{array}
[c]{lll}
-\bigl(|u'|^{p-2}u'\phi_{1}(\sigma)\bigr)'=f(\sigma)\phi_{1}(\sigma)& & \sigma\in
(-\infty,t)
,\\
\\
u\in W^{1,p}_{0}\!
\left(
(-\infty,t),\phi_{1}\right), & & 
\end{array}
\right.
\end{equation}
where $f\ge0$ is such that
\begin{equation}\label{sumf}
\int_{-\infty}^{t} f(\sigma)^{\frac p{p-1}}|\sigma|^p\phi_{1}(\sigma)\,\di\sigma<+\infty.
\end{equation}
Then
\begin{equation}\label{sumder}
\int_{-\infty}^{t} |u'(\sigma)|^{p}|\sigma|^q\phi_{1}(\sigma)\,\di\sigma<+\infty,
\end{equation}
when $0\le q<2$.
\end{proposition}

\begin{proof}
Integrating equation \eqref{2eq} we obtain
\begin{equation*}
|u'(\sigma)|^{p-1}e^{-\frac{\sigma^{2}}2}=\int_{-\infty}^{\sigma}f(\tau)e^{-\frac{\tau^{2}}2}\,\di \tau,
\end{equation*}
then
\begin{equation}\label{sumdera}
\int_{-\infty}^{t} |u'(\sigma)|^{p}|\sigma|^qe^{-\frac{\sigma^{2}}2}\,\di\sigma=M+\int_{-\infty}^{\min\{-1,t\}}\left( \int_{-\infty}^{\sigma}f(\tau)e^{-\frac{\tau^{2}}2}\,\di \tau\right)^{\frac p{p-1}}|\sigma|^qe^{\frac{1}{p-1}\frac{\sigma^{2}}2}\,\di\sigma
\end{equation}
where $M\ge0$ is given by
\begin{equation*}
M=\int_{\min\{-1,t\}}^{t}\left( \int_{-\infty}^{\sigma}f(\tau)e^{-\frac{\tau^{2}}2}\,\di \tau\right)^{\frac p{p-1}}|\sigma|^qe^{\frac{1}{p-1}\frac{\sigma^{2}}2}\,\di\sigma<+\infty.
\end{equation*}
On the other hand, for $\sigma<0$, H\"older inequality gives
\begin{equation}\label{est}
\int_{-\infty}^{\sigma}f(\tau)e^{-\frac{\tau^{2}}2}\,\di \tau\le\left(\int_{-\infty}^{\sigma}f(\tau)^{\frac p{p-1}}|\tau|^pe^{-\frac{\tau^{2}}2}\,\di \tau\right)^{\frac{p-1}p}\left(\int_{-\infty}^{\sigma}\frac1{|\tau|^{p(p-1)}}e^{-\frac{\tau^{2}}2}\,\di \tau\right)^{\frac{1}p},
\end{equation}
then, taking into account the fact that
\begin{equation*}
\int_{-\infty}^{\sigma}\frac1{|\tau|^{p(p-1)}}e^{-\frac{\tau^{2}}2}\,\di \tau\approx\frac{e^{-\frac{\sigma^{2}}2}}{|\sigma|^{1+p(p-1)}},\qquad\text{ as }\sigma\rightarrow-\infty,
\end{equation*}
we have
\begin{equation*}
\left( \int_{-\infty}^{\sigma}f(\tau)e^{-\frac{\tau^{2}}2}\,\di \tau\right)^{\frac p{p-1}}|\sigma|^qe^{\frac{1}{p-1}\frac{\sigma^{2}}2}\approx\frac{|\sigma|^{q}}{|\sigma|^{\frac1{p-1}+p}},\qquad\text{ as }\sigma\rightarrow-\infty.
\end{equation*}
Observing that $\frac1{p-1}+p\ge3$ for $p>1$, we have that for $0\le q<2$ the integral on the right-hand side of \eqref{sumdera} is finite.
\end{proof}
\begin{proof}[Proof of Proposition \ref{half}]
We have already observed that there exists a positive function $u_1=u_1(x_1)$, depending only on $x_1$, which is an eigenfunction corresponding to $\lambda_1(H_t)$. The simplicity of 
 $\lambda_1(H_t)$ can be proven following an argument in \cite{Lin} which we have already adapted in the gaussian context in the proofs of Proposition \ref{prop} (\ref{(b)}) and Proposition \ref{propa} (\ref{(aa)}).
 
 We have also observed that the mapping $t\mapsto \lambda_1(H_t)$ is increasing, so we have only to prove \eqref{shape}. Let us denote by $u$ the positive eigenfunction such that $||u||_{L^p(H_t,\phi_N)}=||u||_{L^p((-\infty,t),\phi_1)}=1$ and let us consider, for $h>0$, the test function
 $$w_1(x)=u(x-h)e^{\frac{hx}p}.$$
 Clearly we have $w_1\in L^p((-\infty,t+h),\phi_1)$ because
 $$\int_{-\infty}^{t+h}w_1(x)^{p}\phi_1(x)\,\di x=\int_{-\infty}^{t+h}\left(u(x-h)^{p}e^{\frac{hx}p}\right)^p\phi_1(x)\,\di x
 =
 e^{\frac{h^{2}}{2}}
 \int_{-\infty}^{t}u(x)^{p}\phi_1(x)\,\di x.
 $$
Taking into account the fact the $w_1'(x)=(u'(x-h)+h\,u(x-h)/p)e^{\frac{hx}p}$ we also have $w_1\in W_0^{1,p}((-\infty,t+h),\phi_1)$, so we can use $w_1$ in the variational characterization 
\eqref{1dim} for $\lambda_1(H_{t+h})$, obtaining
\begin{equation*}
\lambda_1(H_{t+h})
\le
\frac{\dint_{-\infty}^{t+h} |w_1'(x)|^{p}\,\phi_{1}(x)\,\di x}
{\dint_{-\infty}^{t+h} |w_1(x)|^{p}\,\phi_{1}(x)\,\di x}=
\frac{\dint_{-\infty}^{t} \left|u'(x)+u(x)\frac hp\right|^{p}\,\phi_{1}(x)\,\di x}
{\dint_{-\infty}^{t} |u(x)|^{p}\,\phi_{1}(x)\,\di x}.
\end{equation*}
Forming the difference quotient for
$\lambda_1(H_{t})$, we obtain 
\begin{equation*}
\frac{\lambda_1(H_{t+h})-\lambda_1(H_{t})}h
\le
\frac1h\dint_{-\infty}^{t} \left(\left|u'(x)+u(x)\frac hp\right|^{p}-|u'(x)|^{p}\right)\,\phi_{1}(x)\,\di x
\end{equation*}
and, passing to the limit as $h\rightarrow0^+$, we have
$$
\limsup_{h\rightarrow0^+}\frac{\lambda_1(H_{t+h})-\lambda_1(H_{t})}h
\le
\dint_{-\infty}^{t} |u'(x)|^{p-2}u'(x)u(x)\,\phi_{1}(x)\,\di x.
$$

In order to determine the limit as $h \to 0^+$, it is necessary to show that the opposite inequality holds for the  $\liminf\limits_{h \to 0^+}$.

We denote by $u_{h}$ the positive eigenfunction corresponding to $\lambda_1(H_{t+h})$, normalized such that
    \begin{equation}
    \label{norm}
    \|u_{h}\|_{L^{p}((-\infty, t+h), \phi_1)} = 1 \quad \forall h > 0.
\end{equation}
Then we define the ``backwards'' test function
\begin{equation*}
w_2(x)=u_{h}(x+h)e^{-\frac{hx}{p}},
\end{equation*}
which belongs to 
$W_{0}^{1,p}\left( (-\infty ,t),\phi _{1}\right) $ because 
\begin{equation*}
\left\Vert w_2\right\Vert _{L^{p}\left( (-\infty ,t),\phi _{1}\right)
}^{p}=\int_{-\infty }^{t}u_{h}^{p}(x+h)e^{-hx}e^{-\frac{x^{2}}{2}
}\di x
= e^{\frac{h^{2}}{2}}.
\end{equation*}
Similarly
\begin{align*}
\left\Vert w_2^{\prime }\right\Vert _{L^{p}\left( (-\infty ,t),\phi
_{1}\right) }^{p} &=\int_{-\infty }^{t}\left\vert u_{h}^{\prime
}(x+h)-u_{h}(x+h)\frac{h}{p}\right\vert ^{p}e^{-hx}e^{-\frac{x^{2}}{2}}\di x
\\
&=e^{\frac{h^{2}}{2}}\int_{-\infty }^{t+h}\left\vert u_{h}^{\prime
}(\sigma )-u_{h}(\sigma )\frac{h}{p}\right\vert ^{p}\phi _{1}(\sigma
)\di\sigma <+\infty .
\end{align*}
Using $w_2$ in the variational characterization of $\lambda_1(H_t)$, we have
\begin{equation*}
\lambda _{1}(H_{t})\leq \int_{-\infty }^{h}\left\vert u_{h}^{\prime
}(x)-u_{h}(x)\frac{h}{p}\right\vert ^{p}\phi _{1}(x)\di x.
\end{equation*}
Therefore it holds that
\begin{equation}
\frac{\lambda _{1}(H_{t+h})-\lambda _{1}(H_{t})}{h}\geq \int_{-\infty }^{t+h}
\frac{\left\vert u_{h}^{\prime }(x)\right\vert ^{p}-\left\vert
u_{h}^{\prime }(x)-u_{h}(x)\frac{h}{p}\right\vert ^{p}}{h}\phi _{1}(x)\di x
\label{_liminf}
\end{equation}
and hence
\begin{align}
\label{liminf_}
\liminf_{h \to 0^{+}} \frac{\lambda_1(H_{t+h}) - \lambda_1(H_t)}{h}  &
\geq \liminf_{h \to 0^{+}} \int_{-\infty}^{t+h} \frac{|u_{h}'(x)|^p - \left| u_{h}'(x) - u_{h}(x) \frac{h}{p} \right|^p}{h} \phi_1(x) \, \di x \\
&= \lim_{n \to +\infty} \int_{-\infty}^{t+h_{n}} \frac{|u_{h_{n}}'(x)|^p - \left| u_{h_{n}}'(x) - u_{h_{n}}(x) \frac{h_{n}}{p} \right|^p}{h_{n}} \phi_1(x) \, \di x  ,\nonumber
\end{align}
where $\{h_{n}\}_{n \in \mathbb{N}} \subset (0, 1)$ is a suitable sequence converging to zero.

We intend to apply the Lebesgue Dominated Convergence Theorem. To this end, we first seek a function  independent of $h$ and with the suitable integrability, which uniformly bounds the integrand appearing in right-hand side of inequality \eqref{liminf_}.

The Lagrange Theorem ensures that for any pair of real numbers $a$ and $b $, there exists $\eta$ between $|a|$ and $|b|$ such that
\begin{equation}\label{lagr}
    \left| |a|^{p} - |b|^{p} \right| 
    = p \eta^{p-1} \left| |a| - |b| \right| 
    \leq p (|a|^{p-1} + |b|^{p-1}) |a-b|.
\end{equation}
Applying the previous inequality with $a = u_{h}'(x)$ and $b = u_{h}'(x) - u_{h}(x) \dfrac{h}{p}$, we obtain

 \begin{align}
\label{Lagrange}
 \frac{1}{h} 
 \left| 
  \left| u_{h}'(x) \right|^{p} -
   \left| u_{h}'(x) - u_{h}(x)\frac{h}{p} \right|^{p} 
 \right| &\leq
 \left( |u_{h}'(x)|^{p-1} + \left| u_{h}'(x) - u_{h}(x)\frac{h}{p} \right|^{p-1} \right)
 u_{h}(x)   \\
 &\leq C 
 \left( |u_{h}'(x)|^{p-1} u_{h}(x) + \frac{h}{p} u_{h}(x)^{p} \right), \nonumber
\end{align}
where in the last inequality, which can be derived through elementary arguments, we have
\begin{equation*}
    C = \max\{1, 2^{p-2}\}.
\end{equation*}
In what follows, $C$ will denote a positive constant whose value may change from line to line, independent of  $h$ and $x$.

Note that $\lambda_1(H_{t})$ is bounded in a right neighborhood of $t$ due to its monotonicity. Incidentally, as can be straightforwardly shown, $\lambda_1(H_{t})$ is a continuous function too. 
Integrating equation (\ref{1dim}) and then using  H\"older inequality we get
        \begin{align}
    \label{A}
    |u_{h}'(x)|^{p-1} \phi_{1}(x) &=
    \lambda_{1}(H_{t+h}) \int_{-\infty}^{x} u_{h}(\tau)^{p-1} \phi_{1}(\tau) \, \di\tau \\
    &\leq
    \lambda_{1}(H_{t+h}) \left( \int_{-\infty}^{x} u_{h}(\tau)^{p} \phi_{1}(\tau) \, \di\tau \right)^{\frac{p-1}{p}} \left( \int_{-\infty}^{x} \phi_{1}(\tau) \, \di\tau \right)^{\frac{1}{p}} \nonumber \\
    &\leq 
    C \left( \int_{-\infty}^{x} \phi_{1}(\tau) \, \di\tau \right)^{\frac{1}{p}} \leq C \frac{e^{-\frac{x^{2}}{2p}}}{1+|x|^{\frac{1}{p}}}. \nonumber
    \end{align}
Now let us turn our attention to the term $\dfrac{h}{p} u_{h}(x)^p$. From the integral representation of the derivative, we have
    \begin{equation*}
    |u_{h}'(x)| \leq \left( \frac{\lambda_1(H_{t+h})}{\phi_1(x)} \int_{-\infty}^{x} u_{h}(\tau)^{p-1} \phi_1(\tau) \, \di\tau \right)^{\frac{1}{p-1}}.
    \end{equation*}
  It follows that
    \begin{equation*}
    |u_{h}(x)| 
    \leq
    \int_{-\infty}^{x} |u_{h}'(\rho)| \, \di\rho \leq \int_{-\infty}^{x} \left( \frac{\lambda_1(H_{t+h})}{\phi_1(\rho)} \int_{-\infty}^{\rho} u_{h}(\tau)^{p-1} \phi_1(\tau) \, \di\tau \right)^{\frac{1}{p-1}} \, \di\rho.
    \end{equation*}
 By applying H\"older  inequality to the inner integral, taking into account 
 (\ref{norm}), we obtain
    \begin{equation*}
    \left( 
    \frac{1}{\phi_1(\rho)} \int_{-\infty}^{\rho} u_{h}(\tau)^{p-1} \phi_1(\tau) \, \di\tau \right)^{\frac{1}{p-1}} 
    \leq C \left( e^{\frac{\rho^2}{2}} \left( \int_{-\infty}^{\rho} e^{-\frac{\tau^2}{2}} \, \di\tau \right)^{\frac{1}{p}} \right)^{\frac{1}{p-1}}.
    \end{equation*}
    Using the asymptotic behavior of the Gaussian tail, we derive
    \begin{equation*}
    \left( \frac{1}{\phi_1(\rho)} \int_{-\infty}^{\rho} u_{h}(\tau)^{p-1} \phi_1(\tau) \, \di\tau \right)^{\frac{1}{p-1}}
    \leq C \frac{e^{\frac{\rho^2}{2p}}}{(1+|\rho|)^{\frac{1}{p(p-1)}}}.
    \end{equation*}
    From this bound, we infer the following
    \begin{equation*}
    u_{h}(x) \leq C \int_{-\infty}^{x} \frac{e^{\frac{\rho^2}{2p}}}{(1+|\rho|)^{\frac{1}{p(p-1)}}} \, \di\rho \leq C \frac{e^{\frac{x^2}{2p}}}{(1+|x|)^{1 + \frac{1}{p(p-1)}}}.
    \end{equation*}
By raising this to the power $p$ and multiplying by the Gaussian weight $\phi_1(x)$, we conclude that
    \begin{equation}
    \label{B}
    \frac{h}{p}
    u_{h}(x)^p \phi_1(x) 
    \leq \frac{C}{(1+|x|)^{p + \frac{1}{p-1}}}.
    \end{equation}
   Having established the uniform bounds, we proceed to prove the pointwise convergence almost everywhere. 
   
Let $\{{h_n} \}_{n \in \mathbb{N}}$ be the sequence appearing in \eqref{liminf_}. Since $\lambda_1(H_t)$ is a continuous function we have that the sequence 
 $\{u_{h_{n}}\}_{n \in \mathbb{N}}$ is bounded in $W_0^{1,p}((-\infty, t+1), \phi_1)$.
 Consequently, we can find a subsequence 
    $\{h_{n_k}\}_{k \in \mathbb{N}} $    
 and a nonnegative function $\tilde{u} \in W_0^{1,p}((-\infty, t+1), \phi_1)$ such that 

\begin{equation*}
\begin{cases} 
u_{h_{n_{k}}} \to \tilde{u} & \text{weakly in } W_0^{1,p}((-\infty, t+1), \phi_1) \\
\\
u_{h_{n_{k}}} \to \tilde{u} & \text{strongly in } L^{p}((-\infty, t+1), \phi_1) \\
\\
u_{h_{n_{k}}}^{p}(x) \to \tilde{u}^{p}(x) & \text{for a.e. } x \in (-\infty, t+1). 
\end{cases}
\end{equation*}
Note that 
$\tilde{u} \in W_0^{1,p}((-\infty, t),\phi_1)$,
since $\tilde{u}(x)=0$ for any $t\le x \le t+1$.  
Moreover, the Vitali Convergence Theorem ensures that the normalization is preserved in the limit, namely
\begin{equation}
    \label{u_tilde}
  \int_{-\infty}^{t} \tilde{u}(x)^{p} \phi_{1}(x) \, dx
   = 1.
\end{equation}
By the sequential weak lower semicontinuity of the $L^p$ norm we obtain
\begin{equation}
\label{L_lambda}
    \int_{-\infty}^{t} |\tilde{u}'(x)|^{p} \phi_{1}(x) \, dx \leq \liminf_{k\to\infty} \int_{-\infty}^{t+h_{n_k}} |u_{h_{n_k}}'(x)|^{p} \phi_{1}(x) \, dx = \lim_{k\to\infty} \lambda_{1}(H_{t+h_{n_k}}) = \lambda_{1}(H_{t}).
\end{equation}
From \eqref{L_lambda} and \eqref{u_tilde}, it follows that
\begin{equation*}
    \frac{\displaystyle\int_{-\infty}^{t} |\tilde{u}'(x)|^{p} \phi_1(x) \,dx}{\displaystyle\int_{-\infty}^{t} 
    \tilde{u}^{p}(x) \phi_1(x) \,dx} 
    \leq \lambda_1(H_t).
\end{equation*}

By the variational characterization of the first eigenvalue, this Rayleigh quotient must be equal to $\lambda_1(H_t)$, forcing the inequality in \eqref{L_lambda} to be an equality. Since $L^p((-\infty, t+1),\phi_1)$ is uniformly convex, the weak convergence of the sequence
$\{u_{h_{n_{k}}}\}_{k \in \mathbb{N}}$ together with the convergence of the norms implies that $u'_{h_{n_k}} \to \tilde{u}'$ strongly in $L^p((-\infty, t+1),\phi_1)$. Thus, up to a further (not relabeled) subsequence, it holds that
$$
u'_{h_{n_k}}(x) \to \tilde{u}'(x) 
\quad
\text{for a.e.  } x \in (-\infty, t+1).
$$ 
Finally, since the first eigenvalue is simple, the limit function $\tilde{u}$ coincides with the positive eigenfunction, satisfying \eqref{u_tilde}, that we previously denoted by $u$.
\begin{remark}
         The same conclusions could have been reached, in a shorter and more direct way, by arguing directly from the equation \eqref{A}. However, we preferred here to follow a variational approach to clearly separate the properties that depend intrinsically on the specific structure of the equation from those that can be deduced in a more general setting.
   \end{remark}
    The pointwise convergence almost everywhere just established, combined with the integrable dominating functions previously constructed (inequalities \eqref{A} and \eqref{B}), allows us to apply the Dominated Convergence Theorem along the subsequence, finally obtaining
\begin{align*}
\liminf_{h \to 0^{+}} \frac{\lambda_1(H_{t+h}) - \lambda_1(H_t)}{h} &\geq \liminf_{h \to 0^{+}} \int_{-\infty}^{t+h} \frac{|u_{h}'(x)|^p - \left| u_{h}'(x) - u_{h}(x) \frac{h}{p} \right|^p}{h} \phi_1(x) \, \di x \\
&= \lim_{k \to +\infty} \int_{-\infty}^{t+h_{n_k}} \frac{|u_{h_{n_k}}'(x)|^p - \left| u_{h_{n_k}}'(x) - u_{h_{n_k}}(x) \frac{h_{n_k}}{p} \right|^p}{h_{n_k}} \phi_1(x) \, \di x \\
&= \int_{-\infty}^{t} |u'(x)|^{p-2} u'(x) u(x) \phi_1(x) \, \di x .
\end{align*}
The same arguments for $h<0$ provide the full derivative
\begin{equation}\label{shape0}
\frac\di{\di t}\lambda_1(H_{t})=\dint_{-\infty}^{t} |u'(x)|^{p-2}u'(x)u(x)\,\phi_{1}(x)\,\di x.
\end{equation}
It remains to prove that the right-hand side in \eqref{shape0} is equal to the right-hand side in \eqref{shape}.
Indeed, using the equation in \eqref{2eq}, a direct calculation gives
\begin{eqnarray*}
\dint_{-\infty}^{t} |u'(x)|^{p-2}u'(x)u(x)\,\phi_{1}(x)\,\di x=-
\dint_{-\infty}^{t}u'(x) \left(\dint_{-\infty}^{x}|u'(\sigma)|^{p-2}u'(\sigma)\,\phi_{1}(\sigma)\,\di\sigma\right)\di x=\\
=
\dint_{-\infty}^{t}u'(x) \left[-x |u'(x)|^{p-2}u'(x)\phi_{1}(x)+\dint_{-\infty}^{x}\sigma\left(|u'(\sigma)|^{p-2}u'(\sigma)\,\phi_{1}(\sigma)\right)'\,\di\sigma\right]\di x=\\
=
-\dint_{-\infty}^{t}|u'(x)|^{p}x\,\phi_{1}(x)\,\di x+\lambda_1(H_{t})\dint_{-\infty}^{t}|u(x)|^{p}x\,\phi_{1}(x)\,\di x,
\end{eqnarray*}
where the finiteness of the last two integrals is a consequence of Proposition \ref{summab} and the Hardy inequality proved in \cite{BCT}. Finally, we observe that, if $u$ solves \eqref{2eq} then
\begin{equation*}
|u'(x)|^{p}x\,\phi_{1}(x)-\lambda_1(H_{t})|u(x)|^{p}x\,\phi_{1}(x)=
(p-1)\left(|u'(x)|^{p}\,\phi_{1}(x)\right)'+\lambda_1(H_{t})\left(|u(x)|^{p}\,\phi_{1}(x)\right)'
\end{equation*}
and the claim follows.

\end{proof}

\begin{proof}[Proof of Proposition \ref{derQ}]
By Proposition \ref{prop}, (c), it is clear that $Q_p^\sharp(\alpha,t)$ is an increasing function with respect to $t$, so we only need to prove \eqref{shapeQ}. Let us consider, for $0<h<1$, the test function
$$w_1(x)=v_h(x+h)e^{-\frac{2hx+h^2}{2p}},$$
where $v_h$ is the solution to \eqref{2eqQ} with $t$ replaced by $t+h$.
Arguing as in the proof of Proposition \ref{half}, we can use $w_1$ in the variational characterization 
\eqref{1dimQ} for $Q_p^\sharp(\alpha,t)$, obtaining
\begin{align*}
Q_p^\sharp(\alpha,t)&
\ge-\int_{-\infty}^{t}|w_1'(x)|^p\,\phi_{1}(x)\,\di x+\alpha\int_{-\infty}^{t} |w_1(x)|^p\,\phi_{1}(x)\,\di x+p\int_{-\infty}^{t} w_1(x)\,\phi_{1}(x)\,\di x=\\
&=-\int_{-\infty}^{t+h}\left|v_h'(x)-v_h(x)\frac hp\right|^p\,\phi_{1}(x)\,\di x+\alpha\int_{-\infty}^{t+h} |v_h(x)|^p\,\phi_{1}(x)\,\di x\\
&+p\int_{-\infty}^{t+h} v_h(x)e^{\frac{2xh-h^2}{2p}(p-1)}\,\phi_{1}(x)\,\di x
\end{align*}
and
\begin{align}
\label{incr2}
\frac{Q_p^\sharp(\alpha,t+h)-Q_p^\sharp(\alpha,t)}h
&\le
-\frac1h{\dint_{-\infty}^{t+h} \left(|v_h'(x)|^{p}-\left|v_h'(x)-v_h(x)\frac hp\right|^{p}\right)\,\phi_{1}(x)\,\di x}+
   \\
&+\frac phe^{-\frac{h^2}{2p}(p-1)}{\dint_{-\infty}^{t+h} v_h(x)\left(1-e^{\frac{xh}{p}(p-1)}\right)\,\phi_{1}(x)\,\di x}+  \notag
   \\
&+\frac ph\left(1-e^{-\frac{h^2}{2p}(p-1)}\right){\dint_{-\infty}^{t+h}v_h(x)
\,\phi_{1}(x)\,\di x}.  \notag
\end{align}
In order to pass to the limit on the right-hand side of \eqref{incr2}, we observe that, being $v_h$ the solution to problem \eqref{2eqQ} in $(-\infty,t+h)$, it holds
\begin{align*}\dint_{-\infty}^{t+h} |v_h'(x)|^{p}\,\phi_{1}(x)\,\di x=&
\alpha\dint_{-\infty}^{t+h} |v_h(x)|^{p}\,\phi_{1}(x)\,\di x+
\dint_{-\infty}^{t+h} v_h(x)\,\phi_{1}(x)\,\di x\le\\
\le&\left(\frac{\alpha+\varepsilon}{\lambda_1(H_{t+h})}\right)\dint_{-\infty}^{t+h} |v'_h(x)|^{p}\,\phi_{1}(x)\,\di x+C(\varepsilon)
\end{align*}
where $\varepsilon>0$ is sufficiently small and $C(\varepsilon)$ is a constant which depends on $\varepsilon$ only. This means that $\|v_h\|_{W_0^{1,p}(H_{t+h})}$ is bounded, then, recalling that $v_h$ satisfies
\begin{equation}\label{C}
-|v_h'(s)|^{p-2}v_h'(s)\phi_{1}(s)=\int_{-\infty}^{s}\left(\alpha v_h^{p-1}+1\right)\phi_1(x)\,\di x,\qquad s\in
(-\infty,t+h),
\end{equation}
it is possible to repeat the arguments used in the proof of \eqref{tlim} in Proposition \ref{finite} to show that 
\begin{equation*}
\begin{cases} 
v_h \to v & \text{weakly in } W_0^{1,p}((-\infty, t+1), \phi_1) \\
\\
v_h \to v & \text{strongly in } L^{p}((-\infty, t+1), \phi_1) \\
\\
v_h(x) \to v(x) & \text{for a.e. } x \in (-\infty, t+1). 
\end{cases}
\end{equation*}
where $v$ is the solution to problem \eqref{2eqQ} in $(-\infty,t)$.

For the first integral on the right-hand side of \eqref{incr2} we observe that, taking into account \eqref{lagr}, it holds
 \begin{align*}
 \frac{1}{h} 
 \left| 
  \left| v_{h}'(x) \right|^{p} -
   \left| v_{h}'(x) - v_{h}(x)\frac{h}{p} \right|^{p} 
 \right| &\leq
 C' \left( |v_{h}'(x)|^{p-1} + \left| v_{h}'(x) - v_{h}(x)\frac{h}{p} \right|^{p-1} \right)
 v_{h}(x)   \\
 &\leq C 
 \left( |v_{h}'(x)|^{p-1} v_{h}(x) + \frac{h}{p} v_{h}(x)^{p} \right), 
\end{align*}
Using equation \eqref{C} and then using  H\"older inequality we get
        \begin{align*}
    |v_{h}'(x)|^{p-1} \phi_{1}(x) &\le|\alpha|
    \left( \int_{-\infty}^{x} v_{h}(\tau)^{p} \phi_{1}(\tau) \, \di\tau \right)^{\frac{p-1}{p}} \left(\int_{-\infty}^{x}\phi_{1}(\tau) \, \di\tau\right)^{\frac1p}+\int_{-\infty}^{x}\phi_{1}(\tau) \, \di\tau \\
    &\leq 
    C \left( \int_{-\infty}^{x} \phi_{1}(\tau) \, \di\tau \right)^{\frac{1}{p}} +\int_{-\infty}^{x}\phi_{1}(\tau) \, \di\tau,
    \end{align*}
and, as shown in the previous proof we have that $v'_h(x)^p\phi_{1}(x)$ is majorized by a summable function. Similarly, using again \eqref{C}, we have that also $v_h(x)^p\phi_{1}(x)$ is majorized by a summable function, then we get
$$\lim_{h\rightarrow0^+}\frac1h{\dint_{-\infty}^{t+h} \left(|v_h'(x)|^{p}-\left|v_h'(x)-v_h(x)\frac hp\right|^{p}\right)\,\phi_{1}(x)\,\di x}=
\dint_{-\infty}^{t} |v'(x)|^{p-2}v'(x)v(x)\,\phi_{1}(x)\,\di x.
$$
Analogously,
\begin{align*}
&\lim_{h\rightarrow0^+} \frac phe^{-\frac{h^2}{2p}(p-1)}{\dint_{-\infty}^{t+h} v_h(x)\left(1-e^{\frac{xh}{p}(p-1)}\right)\,\phi_{1}(x)\,\di x}=-(p-1)\dint_{-\infty}^{t}x\, v(x)\,\phi_{1}(x)\,\di x,
 \\
&\lim_{h\rightarrow0^+}\frac ph\left(1-e^{-\frac{h^2}{2p}(p-1)}\right){\dint_{-\infty}^{t+h}v_h(x)
\,\phi_{1}(x)\,\di x}=0, 
\end{align*}
where the finiteness of the integral on the right-hand side of the first equality is a consequence of the Hardy inequality proved in \cite{BCT}.
Collecting the above results, we have
\begin{align*}
\limsup_{h\rightarrow0^+}\frac{Q_p^\sharp(\alpha,t+h)-Q_p^\sharp(\alpha,t)}h
\le&-
\dint_{-\infty}^{t} |v'(x)|^{p-2}v'(x)v(x)\,\phi_{1}(x)\,\di x-\\
-&(p-1)\dint_{-\infty}^{t}x\, v(x)\,\phi_{1}(x)\,\di x.
\end{align*}


In order to get the opposite inequality, we consider the test function
$$w_2(x)=v(x-h)e^{\frac{2hx-h^2}{2p}},$$
where $v$ is the solution to \eqref{2eqQ}. 
We can use $w_2$ in the variational characterization 
\eqref{1dimQ} for $Q_p^\sharp(\alpha,t+h)$, obtaining
\begin{align*}
Q_p^\sharp(\alpha,t+h)
&\ge
-\int_{-\infty}^{t+h}|w_2'(x)|^p\,\phi_{1}(x)\,\di x+\alpha\int_{-\infty}^{t+h} |w_2(x)|^p\,\phi_{1}(x)\,\di x+p\int_{-\infty}^{t+h} w_2(x)\,\phi_{1}(x)\,\di x\\
&=-\int_{-\infty}^{t}\left|v'(x)+v(x)\frac hp\right|^p\,\phi_{1}(x)\,\di x+\alpha\int_{-\infty}^{t} |v(x)|^p\,\phi_{1}(x)\,\di x+\\
&+p\int_{-\infty}^{t} v(x)\,\phi_{1}(x)e^{-\frac{2xh+h^2}{2p}(p-1)}\,\di x.
\end{align*}
Then
\begin{align}
\label{incr1}
\frac{Q_p^\sharp(\alpha,t+h)-Q_p^\sharp(\alpha,t)}h
&\ge
-\frac1h{\dint_{-\infty}^{t} \left(\left|v'(x)+v(x)\frac hp\right|^{p}-|v'(x)|^{p}\right)\,\phi_{1}(x)\,\di x}+
   \\
&+\frac phe^{-\frac{h^2}{2p}(p-1)}{\dint_{-\infty}^{t} v(x)\left(e^{-\frac{xh}{p}(p-1)}-1\right)\,\phi_{1}(x)\,\di x}.  \notag\\
&+\frac ph\left(e^{-\frac{h^2}{2p}(p-1)}-1\right)\dint_{-\infty}^{t} v(x)
\,\phi_{1}(x)
\,\di x.  \notag
\end{align}


Arguing as above, it is possible to pass to the limit as $h$ goes to 0 in  \eqref{incr1}, obtaining
\begin{align}
\label{incr11}
\liminf_{h\rightarrow0^+}\frac{Q_p^\sharp(\alpha,t+h)-Q_p^\sharp(\alpha,t)}h
&\ge
-
\dint_{-\infty}^{t} |v'(x)|^{p-2}v'(x)v(x)\,\phi_{1}(x)\,\di x-\\
-&(p-1)\dint_{-\infty}^{t}x\, v(x)\,\phi_{1}(x)\,\di x.
 \notag
\end{align}
Repeating the same argument when $h<0$ we get 
\begin{equation}\label{shapQ}
\frac\partial{\partial t}Q_p^\sharp(\alpha,t)=-\dint_{-\infty}^{t} |v'(x)|^{p-2}v'(x)v(x)\,\phi_{1}(x)\,\di x-(p-1)\dint_{-\infty}^{t}x\, v(x)\,\phi_{1}(x)\,\di x.
\end{equation}
It remains to prove that the right-hand side in \eqref{shapQ} is equal to the right-hand side in \eqref{shapeQ}.
Indeed, using the equation in \eqref{2eqQ}, a direct calculation gives
\begin{eqnarray*}
-\dint_{-\infty}^{t} |v'(x)|^{p-2}v'(x)v(x)\,\phi_{1}(x)\,\di x=
\dint_{-\infty}^{t}v'(x) \left(\dint_{-\infty}^{x}|v'(\sigma)|^{p-2}v'(\sigma)\,\phi_{1}(\sigma)\,\di\sigma\right)\di x=\\
=
\dint_{-\infty}^{t}v'(x) \left[x |v'(x)|^{p-2}v'(x)\phi_{1}(x)-\dint_{-\infty}^{x}\sigma\left(|v'(\sigma)|^{p-2}v'(\sigma)\,\phi_{1}(\sigma)\right)'\,\di\sigma\right]\di x=\\
=
\dint_{-\infty}^{t}|v'(x)|^{p}x\,\phi_{1}(x)\,\di x-\dint_{-\infty}^{t}(\alpha|v(x)|^{p}+v(x))x\,\phi_{1}(x)\,\di x,
\end{eqnarray*}
where the finiteness of the last two integrals is a consequence of Proposition \ref{summab} and of the Hardy inequality proved in \cite{BCT}. Thus,
\begin{equation*}
\frac\partial{\partial t}Q_p^\sharp(\alpha,t)=\dint_{-\infty}^{t}|v'(x)|^{p}x\,\phi_{1}(x)\,\di x-\dint_{-\infty}^{t}(\alpha|v(x)|^{p}+p\,v(x))x\,\phi_{1}(x)\,\di x.
\end{equation*}
Finally, we observe that, if $v$ solves \eqref{2eqQ}, then
\begin{equation*}
|v'(x)|^{p}x\,\phi_{1}(x)-(\alpha|v(x)|^{p}
+p\,v(x))x\,\phi_{1}(x)=
(p-1)\left(|v'(x)|^{p}\,\phi_{1}(x)\right)'+\left((\alpha|v(x)|^{p}+pv(x))\phi_{1}(x)\right)'
\end{equation*}
and the claim follows. 


\end{proof}

\section*{Acknowledgement}

\noindent The research of the  second and third author was partially supported by Italian MIUR  through research projects   PRIN 2022: PRIN20229M52AS Partial differential equations and related geometric-functional inequalities and PRIN PNRR 2022 - P2022YFAJH - Linear and Nonlinear PDE's: New directions and Applications. The research of the  first author was partially supported by Italian MIUR  through research projects  PRIN PNRR 2022 - P2022YFAJH - Linear and Nonlinear PDE's: New directions and Applications.  The first, second and third authors are members of the Gruppo Nazionale per l'Analisi Matematica, la Probabilit\`a
e le loro Applicazioni (GNAMPA) of the Istituto Nazionale di Alta Matematica (INdAM).
The research of the last author was supported in part by NSF Grant DMS-2246817.

\end{document}